\documentclass{amsart}
\pdfoutput=1
\usepackage{hyperref}
\usepackage{amsmath,amssymb,amsthm, tikz}
\usepackage{xcolor}
\usepackage{fancyhdr}
\usepackage{graphicx}
\usepackage{dirtytalk}
\usepackage{amsmath,amssymb,mathrsfs,graphicx,setspace,array,vmargin,fancyhdr}
\usepackage{enumitem}
\usepackage{wrapfig}
\hypersetup{hidelinks}

\newtheorem{theorem}{Theorem}[section]
\newtheorem{lemma}[theorem]{Lemma}
\newtheorem{corollary}[theorem]{Corollary}

\theoremstyle{definition}
\newtheorem{definition}[theorem]{Definition}
\newtheorem{example}[theorem]{Example}

\newtheorem{remark}[theorem]{Remark}

\newtheorem{thmintro}{Theorem}

\newtheorem{corintro}[thmintro]{Corollary}

\begin{document}
\date{\today}
\title{Enumerating Integral Points on Triangles in $\mathbb{R}^3$}
\author[M. Aldrich]{Max Aldrich}
\address{Adams State University, 208 Edgemont Blvd, Alamosa, CO 81101, USA}
\email{maxaldrich72@gmail.com}
\author[D. Gulbrandsen]{Daniel Gulbrandsen}
\address{Adams State University, 208 Edgemont Blvd, Alamosa, CO 81101, USA}
\email{dangulbrandsen@yahoo.com}
\author[C. Keddy]{Clara Keddy}
\address{Adams State University, 208 Edgemont Blvd, Alamosa, CO 81101, USA}
\email{clarakeddy5@gmail.com}
\begin{abstract}
    We present two ways of constructing  Ehrhart polynomials corresponding to certain triangles in $\mathbb{R}^3$. One of our methods employs Popoviciu's Theorem, while the other method uses a version of Pick's Theorem adapted to rank-2 lattices in $\mathbb{R}^3$. We then apply our methods to linear Diophantine equations of the form $ax+by+cz=n$ and develop upper and lower bounds on the number of solutions in nonnegative integers $x$, $y$, and $z$. The bounds we develop are closed form quasi-polynomials, and the difference between them is bounded by a constant depending only on $a$ and $b$. As an application, we prove that the $k$th Frobenius number exists for any triple of positive integers and any $k$.
\end{abstract}

\maketitle

\tableofcontents

\section{Introduction}

At the heart of this article, we are interested in the number of solutions in nonnegative integers $x$, $y$, and $z$ to linear Diophantine equations of the form 
\begin{equation}\label{eqn:1}
    ax+by+cz=n,
\end{equation}
for positive integers $a$, $b$, $c$, and $n$. For a concrete example, a person with $x$ nickels, $y$ dimes, and $z$ quarters has a total of $5x+10y+25z$ cents. Another person can have a different combination of nickels, dimes, and quarters that also sum to a total of $5x+10y+25z$ cents. It is therefore reasonable to ask how many different combinations of nickels, dimes, and quarters sum to the same number $n$, for various values of $n$.

This problem has origins in \cite{Sylvester1882}, where Sylvester considered the number of solutions to equations of the form $ax+by=n$. Therein it is shown that, for relatively prime positive integers $a$ and $b$, the largest integer $n$ for which $ax+by=n$ has no solutions in nonnegative integers $x$ and $y$ is given by $n=ab-a-b$. Then, in \cite{sylvester1884problem}, Sylvester posed his theorem as a recreational problem. This appears to have captured the interest of Frobenius who is often credited with popularizing similar questions during their lectures (\cite{Schur1926}, \cite{Brauer1942}). In particular, Frobenius considered linear Diophantine equations with arbitrarily many terms, and asked for the largest value of $n$ for which $a_1x_1+a_2x_2+\cdots+a_kx_k=n$ has no solutions in nonnegative integers. Such an integer $n$ is now called the {\it Frobenius number} of $a_1,\ldots,a_k$. For a general reference on the Frobenius number, consult \cite{ramirez2005diophantine}. Formulas that compute the Frobenius number for three terms can be found in \cite{fang2024graphical}.

This paper, in some ways, addresses the opposite extreme  case. That is, we consider values of $n$ for which Equation \ref{eqn:1} has many solutions. Note that, since $a,b,c>0$, the question of how many nonnegative integer solutions are there to Equation \ref{eqn:1} is equivalent to asking how many points with nonnegative integer-valued coordinates (integral points) are contained on the plane with equation $ax+by+cz=n$. This allows for geometric techniques to be used. One in particular is {\sl Pick's Theorem}, which relates the number of integral points contained on a polygon to the polygon's area \cite{pick1899geometrisches}. However, the scope of Pick's Theorem is limited to polygons in $\mathbb{R}^2$, and we often consider polygons in higher dimensions. Formulas that count the number of integral points contained on more general polyhedra were later developed by Ehrhart \cite{ehrhart1962polyedres}. In the case of a polygon in $\mathbb{R}^3$, an adjustment to Pick's theorem is necessary. Essentially, the area term must be replaced by a term that, in some sense, accounts for the fact that distances between integral points may have changed. See Lemma \ref{lem:Picks in 3D} below. 

The tools employed by Ehrhart, while now considered classical, can present a significant hurdle for students just beginning to learn the theory. Beck and Robins in \cite{beck2015computing} and Barvinok in \cite{Barvinok1994Algorithm}, \cite{Barvinok1994Ehrhart}, and especially \cite{barvinok2008integer} modernize the notation some  and provide pivotal algorithms that simplify the process. We provide methods which are accessible to anyone who has taken a basic linear algebra course, and does not mind learning a few basic facts about modular arithmetic. Although our results are mostly limited to polygons in $\mathbb{R}^3$, we believe the contents of this paper highlight a key example in the theory, and constitute an ideal setting for beginners of computational geometry.
\medskip

In Chapter \ref{sec:prelim}, we provide the necessary background material on lattices, restricted partition functions, and Ehrhart Theory. Readers familiar with these topics may skip this chapter, using it as a reference when needed. Chapter \ref{sec:main results} contains the main results. In section \ref{subsec:axial triangles}, we provide two proofs of the following geometric fact.  
\begin{thmintro}\label{thm:A}
Let $T\subset\mathbb{R}^3$ be a triangle with vertices $(a,0,0)$, $(0,b,0)$, and $(0,0,c)$, for positive integers $a$, $b$, and $c$. Then the number of points in $T\cap\mathbb{Z}^3$ is given by the expression
\[
    \frac{\gcd(bc,ac,ab)}{2}+\frac{{\gcd(a,b)+\gcd(a,c)+\gcd(b,c)}}{2}+1.
\]
\end{thmintro}
\noindent The first proof is entirely combinatorial and does not seem to appear in the literature, although we rely on techniques found in most standard textbooks on the subject. The second proof is more geometric and follows from classical results already present in \cite{pick1899geometrisches} and \cite{ehrhart1962polyedres}. There is also a third proof, which can be obtained by applying Ehrhart's Theorem for rational polytopes \cite{ehrhart1962polyedres}. However, the machinery needed to do this proof justice would take us too far afield from our primary goals, and so, the details are left to the reader. 
\medskip

Given the three points $(a,0,0)$, $(0,b,0)$, and $(0,0,c)$ that define the triangle $T$ from Theorem \ref{thm:A}, we can leverage tools from linear algebra to get the equation of the plane conaining $T$: $bcx+acy+abz=abc$. Clearly, point on $T$ is a solution to this equation.

\begin{example}
   A method attributed to early Egyptians is to write the reciprocal of a positive integer $n$ as the sum of three unit fractions: $\frac{1}{a}+\frac{1}{b}+\frac{1}{c}=\frac{1}{n}$, for positive integers $a$, $b$, and $c$. Multiplying both sides of this equation by $abcn$ gives us $bcn+acn+abn=abc$. It follows that $\frac{1}{n}$ can be written as $\frac{1}{a}+\frac{1}{b}+\frac{1}{c}$ if and only if the point $(n,n,n)$ is on $T$.  
\end{example}

As a corollary of Theorem \ref{thm:A}, we get the following number-theoretic result.
\begin{corintro}\label{cor:B}
    Let $a$, $b$, and $c$ be positive integers that all divide $n$. Then the number of solutions in nonnegative integers to $ax+by+cz=n$
    is given by the expression
    \[
        \frac{\gcd(n^2/bc,n^2/ac,n^2/ab)}{2}+\frac{\gcd(n/a,n/b)+\gcd(n/a,n/c)+\gcd(n/b,n/c)}{2}+1.
    \]
\end{corintro}
We use Corollary \ref{cor:B} to describe the Ehrhart polynomials of certain triangles in $\mathbb{R}^3$ (see Remark \ref{rmk:Ehrhart Polynomials} below). Since our first proof of Theorem \ref{thm:A} is independent of both Pick's Theorem and Ehrhart's construction, we thus arrive at a new derivation of these Ehrhart polynomials. Our results have some overlap with results in \cite{Ion2013}. Therein Iona{\c{s}}cu computes the Ehrhart polynomial of any equilateral triangle with vertices in $\mathbb{Z}^3$. In particular, both our results and those \cite{Ion2013} can be used to derive the Ehrhart polynomial when $a=b=c$. Note that the question of which polynomials could possibly be Ehrhart polynomials for triangles in lattices is addressed in \cite{hofscheier2018ehrhart}.

For another interpretation of the number of solutions to Equation \ref{eqn:1}. Consider three bins labeled A, B, and C, and suppose we have $n$ indistinguishable objects that we must distribute among the three bins where bin A only accepts groups of size $a$, bin B only accepts groups of size $b$, and bin C only accepts groups of size $c$. Assume first that not every bin is required to contain any objects. If $x$, $y$, and $z$ denote the numbers of groups that have been placed in bin A, bin B, and bin C, respectively, then the number of solutions in nonnegative integers to $ax+by+cz=n$ is equal to the total number of ways the $n$ objects could be distributed among the three bins. If  we instead require each bin to contain at least one object, then the total number of ways the $n$ objects could be distributed is equal to the number of solutions in positive integers to Equation \ref{eqn:1}.
\begin{example}\label{example:1}
    Let $B(n)$ denote the number of ways we can distribute $n$ objects among three bins without any requirement on how many objects must be grouped together to be placed in a bin. Clearly, $B(n)$ equals the number of solutions in nonnegative integers to $x+y+z=n$. Thus, by Corollary \ref{cor:B}, 
    \[
        B(n)=\frac{n^2}{2}+\frac{3n}{2}+1.
    \]
\end{example}
In Section \ref{subsec:bounds}, we then adapt the method used in proving Theorem \ref{thm:A} to produce upper and lower bounds to the number of solutions to $ax+by+cz=n$. While the proof assumes $a$, $b$, and $c$ to be pairwise relatively prime, it is explained in the paragraphs preceding the proof that this does not come at the cost of any generality.
\begin{thmintro}\label{thm:C}
    Given pairwise relatively prime positive integers $a$, $b$, and $c$, let $P(n)$ denote the number of solutions in nonnegative integers to $ax+by+cz=n$, for some positive integer $n$. Then there are quasi-polynomials $U(n)$ and $L(n)$ such that, for all positive integers $n$,
    \begin{enumerate}[label={\roman*.}\;]
        \item $L(n)\le P(n)\le U(n)$;
        \item $U(n)-L(n)\le\frac{a+b}{4}$; and
        \item $U(n)$ and $L(n)$ both have period $abc$.
    \end{enumerate}
\end{thmintro}
As an application of Theorem \ref{thm:C}, we show that the {\sl generalized Frobenius number} exists for any choice of positive integers $a,b,c$ (Remark \ref{rmk:Frob number}).  

There are a number of known approximations for the number $P(n)$ of nonnegative solutions to $ax+by+cz=n$, the earliest of which dates back to Sylvester: $P(n)\approx\frac{n^2}{2abc}$. More recently, one can find upper and lower bounds in \cite{Binner2021Combinatorial}, the Thesis of Damanvir Singh Binner. Moreover, it was shown in \cite{LiCh07} that $P(n)$ is evaluated by an integer-valued quasi-polynomial. Additionally, algorithms exist that compute the number of solutions in the more general case \cite{Liso95}, \cite{Barvinok1994Algorithm}. In \cite{Beck_1999}, a closed form solution is given, but computing it is not easy, and deriving the formula requires significant work. 
\medskip

\subsection{Acknowledgments}

The authors would like to express their gratitude for the Department of Mathematics at Adams State University and their enthusiastic support throughout this project. Particular gratitude is owed to Meredith Anderson, Kazi Arefin, and Preston Keith for suggesting improvements to earlier drafts. Additionally, the authors are indebted to Violeta Vasilevska, whose prior mentorship of the second author sparked the initial ideas for this project.

\section{Preliminaries}\label{sec:prelim}

Our definitions and conventions in notation are slight modifications of those found in \cite{barvinok2008integer} and \cite{beck2015computing}. All spaces considered are subspaces of $\mathbb{R}^k$ with the standard Euclidean metric. We denote by $\mathbb{Z}_{\ge0}$ and $\mathbb{Z}_{>0}$ the nonnegative and positive integers, respectively. A point with integer-valued coordinates is called an {\it integral point}. In other words, the point $(x_1,\ldots,x_k)$ is integral if $x_i\in\mathbb{Z}$ for all $i$. 

Vectors will be denoted using arrows, such as $\vec{v}$. When writing vectors coordinate-wise, we will at times use angled brackets to distinguish them from points, and at other times denote (position) vectors by their terminal points when the distinction is unnecessary. Given a vector space $V\subset\mathbb{R}^k$, we write $dim(V)=d$ if $V$ contains a set of $d$ linearly independent vectors but every set of $d+1$ vectors in $V$ is linearly dependent. Given vectors $\vec{u}$ and $\vec{v}$ in $\mathbb{R}^k$, we will write $\vec{u}\cdot\vec{v}$ for their standard inner product. When $k=3$, we write $\vec{u}\times\vec{v}$ to mean the cross product that one learns in a standard calculus course.

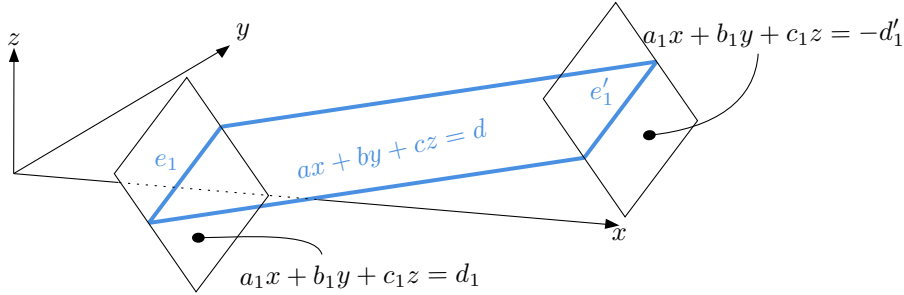
\begin{figure}[htp]
    \centering
    \begin{tikzpicture}[x=0.6pt,y=0.45pt,yscale=-1,xscale=1] 
    \draw    (248.46,176.48) -- (435.38,197.81) ;
    \draw [shift={(438.36,198.15)}, rotate = 186.51] [fill={rgb, 255:red, 0; green, 0; blue, 0 }  ][line width=0.08]  [draw opacity=0] (8.93,-4.29) -- (0,0) -- (8.93,4.29) -- cycle    ;
    \draw    (58.56,154.81) -- (58.77,52.54) ;
    \draw [shift={(58.78,49.54)}, rotate = 90.12] [fill={rgb, 255:red, 0; green, 0; blue, 0 }  ][line width=0.08]  [draw opacity=0] (8.93,-4.29) -- (0,0) -- (8.93,4.29) -- cycle    ;
    \draw    (58.56,154.81) -- (191.88,49.26) ;
    \draw [shift={(194.24,47.39)}, rotate = 141.63] [fill={rgb, 255:red, 0; green, 0; blue, 0 }  ][line width=0.08]  [draw opacity=0] (8.93,-4.29) -- (0,0) -- (8.93,4.29) -- cycle    ;
    \draw [color={rgb, 255:red, 74; green, 144; blue, 226 }  ,draw opacity=1 ][line width=1.5]    (143.3,196.08) -- (416.26,141.6) ;
    \draw [color={rgb, 255:red, 74; green, 144; blue, 226 }  ,draw opacity=1 ][line width=1.5]    (143.3,196.08) -- (188.75,115.62) ;
    \draw [color={rgb, 255:red, 74; green, 144; blue, 226 }  ,draw opacity=1 ][line width=1.5]    (416.26,141.6) -- (461.7,61.15) ;
    \draw [color={rgb, 255:red, 74; green, 144; blue, 226 }  ,draw opacity=1 ][line width=1.5]    (188.75,115.62) -- (461.7,61.15) ;
    \draw  [dash pattern={on 0.84pt off 2.51pt}]  (248.46,176.48) -- (124.27,162.24) ;
    \draw   (436.06,11.75) -- (487.33,110.54) -- (441.89,191) -- (390.62,92.2) -- cycle ;
    \draw    (124.27,162.24) -- (58.56,154.81) ;
    \draw   (167.1,74.41) -- (218.37,173.21) -- (172.93,253.67) -- (121.66,154.87) -- cycle ;
    \draw    (457.15,122.7) .. controls (498.63,117.04) and (523.31,93.83) .. (525.36,54.54) ;
    \draw [shift={(457.15,122.7)}, rotate = 352.24] [color={rgb, 255:red, 0; green, 0; blue, 0 }  ][fill={rgb, 255:red, 0; green, 0; blue, 0 }  ][line width=0.75]      (0, 0) circle [x radius= 3.35, y radius= 3.35]   ;
    \draw    (174.36,208.42) .. controls (200.41,208.13) and (247.71,202.77) .. (251.82,222.41) ;
    \draw [shift={(174.36,208.42)}, rotate = 359.35] [color={rgb, 255:red, 0; green, 0; blue, 0 }  ][fill={rgb, 255:red, 0; green, 0; blue, 0 }  ][line width=0.75]      (0, 0) circle [x radius= 3.35, y radius= 3.35]   ;
    \draw (438.36,198.32) node [anchor=north] [inner sep=0.75pt]   [align=left] {$\displaystyle x$};
    \draw (196.44,45.68) node [anchor=south west] [inner sep=0.75pt]   [align=left] {$\displaystyle y$};
    \draw (58.78,49.37) node [anchor=south] [inner sep=0.75pt]   [align=left] {$\displaystyle z$};
    \draw (293.6,124.42) node [anchor=north] [inner sep=0.75pt]  [color={rgb, 255:red, 74; green, 144; blue, 226 }  ,opacity=1 ,rotate=-347.22] [align=left] {$\displaystyle ax+by+cz=d$};
    \draw (535,24.34) node [anchor=north] [inner sep=0.75pt]  [color={rgb, 255:red, 0; green, 0; blue, 0 }  ,opacity=1 ,rotate=-359.25] [align=left] {$\displaystyle a_{1} x+b_{1} y+c_{1} z=-d'_{1}$};
    \draw (275.86,228.83) node [anchor=north] [inner sep=0.75pt]  [color={rgb, 255:red, 0; green, 0; blue, 0 }  ,opacity=1 ,rotate=-359.25] [align=left] {$\displaystyle a_{1} x+b_{1} y+c_{1} z=d_{1}$};
    \draw (163.76,154.13) node [anchor=south east] [inner sep=0.75pt]  [font=\normalsize,color={rgb, 255:red, 74; green, 144; blue, 226 }  ,opacity=1 ] [align=left] {$\displaystyle e_{1}$};
    \draw (436.71,99.66) node [anchor=south east] [inner sep=0.75pt]  [font=\normalsize,color={rgb, 255:red, 74; green, 144; blue, 226 }  ,opacity=1 ] [align=left] {$\displaystyle e'_{1}$};
    \end{tikzpicture}
    \caption{A parallelogram in $\mathbb{R}^3$.}\label{fig:parallelogram}
\end{figure}
A {\it polyhedron} $P\subset V$ is defined by finitely many linear inequalities:
\[
    P=\left\{(x_1,\ldots,x_k)\in V\;\;|\;\;a_{i1}x_1+\cdots+a_{ik}x_k\le d_i, \;\;i\in I\right\}
\]
for some (possibly empty) finite set $I$, where $a_{ij}$ and $d_i$ are real numbers for all $i\in I$ and $1\le j\le k$. The {\it boundary} of $P$, denoted $\partial P$, is the set of all points $(x_1,\ldots,x_k)$ where equality holds for some $i\in I$. For example, a parallelogram $P$ in $\mathbb{R}^3$ is described by six inequalities. Two are needed to identify the plane containing $P$, and four to identify the sides of $P$. If the plane containing $P$ has equation $ax+by+cz=d$, then the six defining inequalities for $P$ have forms
\begin{align*}
        ax+by+cz\le d \hspace{1.4cm}&\hspace{1.4cm}-ax-by-cz\le -d \\ 
         a_1x+b_1y+c_1z\le d_1 \hspace{1cm}&\hspace{1cm}-a_1x-b_1y-c_1z\le d_1' \\
          a_2x+b_2y+c_2z\le d_2 \hspace{1cm}&\hspace{1cm}-a_2x-b_2y-c_2z\le d_2',
\end{align*}
where, for $i\in\{1,2\}$, $e_i$ and $e_i'$ denote a pair of opposite edges of $P$, and $a_ix+b_iy+c_iz=d_i$ is the equation of a plane containing $e_i$ that is orthogonal to $P$. See Figure \ref{fig:parallelogram}.
\medskip

We  will assume all polyhedra to be bounded and simple (without holes) unless stated otherwise. A polyhedron, all of whose vertices are integral points, is called {\it integral}. A 2-dimensional polyhedron is also called a {\it polygon}. For a general polyhedron $P$, we denote by $Vol(P)$ the $n$-dimensional Euclidean volume of $P$. If $P$ is a polygon, we often write $Area(P)$ instead.
\medskip

A primary goal throughout is to determine the number of integral points contained on polyhedra. Given a vector space $V\subseteq\mathbb{R}^k$, and a polyhedron $P\subset V$, we write $|P\cap\mathbb{Z}^k|$ to denote the number of integral points contained on $P$. Similarly, $|\partial P\cap\mathbb{Z}^k|$ denotes the number of integral points contained on the boundary of $P$. It will also be necessary to consider the points in $P$ not on $\partial P$. We define the {\it interior} of $P$ as $\overset{\circ}{P}=P\setminus\partial P$. With this notation, the number of integral points in the interior of $P$ can be expressed as $\left|\overset{\circ}{P}\cap\mathbb{Z}^k\right|$.
\medskip

We recall a couple of facts particular to the case $k=3$ and $dim(V)=2$. Let $\vec{u}=\langle u_1,u_2,u_3\rangle$ and $\vec{v}=\langle v_1,v_2,v_3\rangle$ be linearly independent vectors in $\mathbb{R}^3$. The equation of the plane spanned by $\vec{u}$ and $\vec{v}$ is found by first computing the cross product $\vec{u}\times\vec{v}$, then setting 
\[
    \vec{u}\times\vec{v}\,\cdot\,(\vec{x}-\vec{p})=0,
\]
where $\vec{x}=\langle x,y,z\rangle$, and $\vec{p}$ is any position vector whose terminal point is on the plane. The area of the parallelogram $P$ with sides $\vec{u}$ and $\vec{v}$ is computed as $Area(P)=|\vec{u}\times\vec{v}|$, the determinant of their cross product.

\subsection{Lattices and Polygons in \texorpdfstring{$\mathbb{R}^3$}{}}

For a general treatment of lattices in $\mathbb{R}^k$ and integral points in polyhedra, we refer the reader to \cite{barvinok2008integer}.
\begin{definition}
    Let $V\subset\mathbb{R}^k$ be a vector space. A subset $\Lambda\subset V$ is called a {\it lattice} if 
    \begin{enumerate}[label={\roman*.}\;]
         \item $\Lambda$ is an additive subgroup of $V$: for any $\vec{u},\vec{v}\in\Lambda$, $\vec{u}\pm\vec{v}\in\Lambda$;
         \item $\Lambda$ is {\it discrete}: for any bounded subset $B\subset V$, $B\cap\Lambda$ is finite; and
         \item $\Lambda$ spans $V$.
     \end{enumerate}
\end{definition}
The {\it rank} of a lattice $\Lambda$ is the dimension of $V$. If the rank of $\Lambda$ is $k$, we call $\Lambda$ a {\it full rank} lattice. We refer to $\mathbb{Z}^k$ as the {\it standard full rank lattice} in $\mathbb{R}^k$. Let $dim(V)=d$. A {\it basis} for $\Lambda$ is a collection of $d$ vectors, $\vec{u}_1,\ldots,\vec{u}_d$, such that every point $\vec{u}\in\Lambda$ admits a unique representation
\[
    \vec{u}=x_1\vec{u}_1+\cdots+x_d\vec{u}_d,
\]
where $x_i\in\mathbb{Z}$ for $i=1,\ldots,d$. That bases exist for lattices is proved by ~\cite[Theorem 10.4]{barvinok2008integer}. Note that uniqueness of representations implies $\vec{u}_1,\ldots,\vec{u}_d$ are linearly independent vectors. However, not all collections of $d$ linearly independent vectors in $\Lambda$ form a basis.

A {\it sublattice} of $\Lambda$ is obtained by taking the spanning set of a collection of linearly independent vectors in $\Lambda$. We emphasize that a sublattice is itself a lattice for which the spanning vectors give a basis.
\medskip

For the remainder, we will focus exclusively on {\it integer lattices}, lattices of the form $\Lambda=V\cap\mathbb{Z}^k$, where $dim(V)\le k$. Such lattices are {\it maximal}, meaning there does not exist another integer lattice in $V$ containing $\Lambda$ as a proper sublattice. Our main source of examples arise when $V$ is a plane through the origin in $\mathbb{R}^3$, and $\Lambda=V\cap\mathbb{Z}^3$ (see Figure \ref{fig:a rank 2 lattice}). We first record some of the key definitions and facts concerning lattices without any restriction on their rank. 
\medskip

\begin{definition}
    Given a lattice $\Lambda\subset V$ with basis vectors $\vec{u}_1,\ldots,\vec{u}_d$, a {\it fundamental parallelepiped} of $\Lambda$ has the form
    \[
        \Pi=\left\{\sum_{i=1}^dx_i\vec{u}_i\;\left|\;0\le x_i\le1 \;\; {\rm for} \;\; i=1,\ldots,d\right.\right\}.
    \]
    The {\it determinant} of $\Lambda$, denoted $det(\Lambda)$, is defined to be the volume of $\Pi$, that is, $det(\Lambda)=Vol(\Pi)$.
\end{definition}
That $det(\Lambda)$ is well-defined follows from ~\cite[Theorem 10.8]{barvinok2008integer}.
Note that the interior of a fundamental parallelepiped $\Pi$ has empty intersection with $\Lambda$. Indeed, suppose $\vec{n} \in\overset{\circ}{\Pi}\cap\Lambda$. Then $\vec{n}$ admits a unique decomposition $\vec{n}=x_1\vec{u}_1+\cdots x_d\vec{u}_d$, where $\vec{u}_1,\ldots,\vec{u}_d$ is a basis for $\Lambda$, and $x_i\in\mathbb{Z}$ for $i=1,\ldots,d$. Since $\vec{n}\in\Pi$, it follows that $x_i$ is either 0 or 1 for $i=1,\ldots,d$. But then $\vec{n}\in\partial\Pi$, a contradiction since $\partial\Pi\cap\overset{\circ}{\Pi}=\emptyset$. Thus $|\overset{\circ}{\Pi}\cap\Lambda|=0$.

Given a sublattice $\Lambda'\subseteq\Lambda$, the {\it index} of $\Lambda'$ in $\Lambda$ is denoted by $[\Lambda':\Lambda]$ and defined by
\[
    [\Lambda':\Lambda]=\frac{det(\Lambda')}{det(\Lambda)}.   
\]
Let $\vec{v}_i$ be a collection of $d$ vectors in $\Lambda$, and define the {\it half-open parallelepiped} $\Pi_{1/2}$ by 
$$\Pi_{1/2} = \left\{ \sum_{i=1}^d x_i \vec{v}_i \mid 0 \le x_i < 1 \text{ for } i=1, \dots, d \right\}.$$ 
Note that the collection $\vec{v}_1,\ldots,\vec{v}_d$ spans a sublattice $\Lambda'$ of $\Lambda$ for which $\Pi$ is a fundamental parallelepiped. We record the following facts whose derivations can be found in section 1.2 of \cite{cassels1997introduction}. We will use these facts to establish a criterion for when a parallelepiped in $\Lambda$ is fundamental. 
\begin{enumerate}[label={\roman*.}\;]
    \item $|\Pi_{1/2}\cap\Lambda|=[\Lambda':\Lambda]$; and
    \item $\Lambda\subseteq\Lambda'\subseteq[\Lambda':\Lambda]\Lambda$.
\end{enumerate}
First note that $\Pi$ is fundamental precisely when the only points in $\Pi\cap\Lambda$ are the vertices of $\Pi$. In this case,  $\Pi_{1/2}\cap\Lambda$ contains one point, the origin. Supposing this is the only point in $\Pi_{1/2}\cap\Lambda$, it follows that $det(\Lambda)=det(\Lambda')$ by fact i. Then, by fact ii, it follows that $\Lambda=\Lambda'$. We thus arrive at the following criterion for when a parallelepiped in $\Lambda$ is fundamental.
\begin{lemma}\label{lem:simple is fundamental}
    Let $\Pi$ be a parallelepiped in $\Lambda$ that intersects $\Lambda$ only in the vertices of $\Pi$. Then $\Pi$ is a fundamental parallelepiped for $\Lambda$. 
\end{lemma}

We return to the case $k=3$ and $dim(V)=2$. In this case, the vector space $V$ spanned by $\Lambda$ is a plane with an equation of the form $ax+by+cz=0$, for $a,b,c\in\mathbb{R}$. See Figure \ref{fig:a rank 2 lattice}.
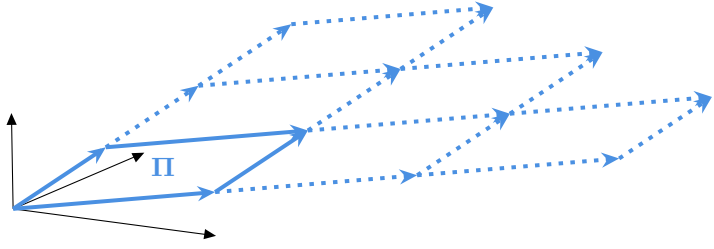
\begin{figure}[htp]
    \centering
    \begin{tikzpicture}[x=0.5pt,y=0.45pt,yscale=-1,xscale=1]
    \draw    (70.25,184.43) -- (219.97,206.56) ;
    \draw [shift={(222.94,207)}, rotate = 188.41] [fill={rgb, 255:red, 0; green, 0; blue, 0 }  ][line width=0.08]  [draw opacity=0] (8.93,-4.29) -- (0,0) -- (8.93,4.29) -- cycle    ;
    \draw    (70.25,184.43) -- (69.05,107.31) ;
    \draw [shift={(69,104.31)}, rotate = 89.11] [fill={rgb, 255:red, 0; green, 0; blue, 0 }  ][line width=0.08]  [draw opacity=0] (8.93,-4.29) -- (0,0) -- (8.93,4.29) -- cycle    ;
    \draw    (70.25,184.43) -- (166.3,138.75) ;
    \draw [shift={(169.01,137.46)}, rotate = 154.57] [fill={rgb, 255:red, 0; green, 0; blue, 0 }  ][line width=0.08]  [draw opacity=0] (8.93,-4.29) -- (0,0) -- (8.93,4.29) -- cycle    ;
    \draw [color={rgb, 255:red, 74; green, 144; blue, 226 }  ,draw opacity=1 ][line width=1.5]    (70.25,184.43) -- (218.26,170.76) ;
    \draw [shift={(222.24,170.39)}, rotate = 174.72] [fill={rgb, 255:red, 74; green, 144; blue, 226 }  ,fill opacity=1 ][line width=0.08]  [draw opacity=0] (13.4,-6.43) -- (0,0) -- (13.4,6.44) -- (8.9,0) -- cycle    ;
    \draw [color={rgb, 255:red, 74; green, 144; blue, 226 }  ,draw opacity=1 ][line width=1.5]    (70.25,184.43) -- (136.8,135.34) ;
    \draw [shift={(140.01,132.96)}, rotate = 143.59] [fill={rgb, 255:red, 74; green, 144; blue, 226 }  ,fill opacity=1 ][line width=0.08]  [draw opacity=0] (13.4,-6.43) -- (0,0) -- (13.4,6.44) -- (8.9,0) -- cycle    ;
    \draw [color={rgb, 255:red, 74; green, 144; blue, 226 }  ,draw opacity=1 ][line width=1.5]    (140.01,132.96) -- (288.03,119.29) ;
    \draw [shift={(292.01,118.93)}, rotate = 174.72] [fill={rgb, 255:red, 74; green, 144; blue, 226 }  ,fill opacity=1 ][line width=0.08]  [draw opacity=0] (13.4,-6.43) -- (0,0) -- (13.4,6.44) -- (8.9,0) -- cycle    ;
    \draw [color={rgb, 255:red, 74; green, 144; blue, 226 }  ,draw opacity=1 ][line width=1.5]  [dash pattern={on 1.69pt off 2.76pt}]  (222.24,170.39) -- (370.25,156.72) ;
    \draw [shift={(374.24,156.35)}, rotate = 174.72] [fill={rgb, 255:red, 74; green, 144; blue, 226 }  ,fill opacity=1 ][line width=0.08]  [draw opacity=0] (13.4,-6.43) -- (0,0) -- (13.4,6.44) -- (8.9,0) -- cycle    ;
    \draw [color={rgb, 255:red, 74; green, 144; blue, 226 }  ,draw opacity=1 ][line width=1.5]  [dash pattern={on 1.69pt off 2.76pt}]  (292.01,118.93) -- (440.02,105.26) ;
    \draw [shift={(444,104.89)}, rotate = 174.72] [fill={rgb, 255:red, 74; green, 144; blue, 226 }  ,fill opacity=1 ][line width=0.08]  [draw opacity=0] (13.4,-6.43) -- (0,0) -- (13.4,6.44) -- (8.9,0) -- cycle    ;
    \draw [color={rgb, 255:red, 74; green, 144; blue, 226 }  ,draw opacity=1 ][line width=1.5]  [dash pattern={on 1.69pt off 2.76pt}]  (374.24,156.35) -- (522.25,142.69) ;
    \draw [shift={(526.23,142.32)}, rotate = 174.72] [fill={rgb, 255:red, 74; green, 144; blue, 226 }  ,fill opacity=1 ][line width=0.08]  [draw opacity=0] (13.4,-6.43) -- (0,0) -- (13.4,6.44) -- (8.9,0) -- cycle    ;
    \draw [color={rgb, 255:red, 74; green, 144; blue, 226 }  ,draw opacity=1 ][line width=1.5]    (222.24,170.39) -- (288.79,121.3) ;
    \draw [shift={(292.01,118.93)}, rotate = 143.59] [fill={rgb, 255:red, 74; green, 144; blue, 226 }  ,fill opacity=1 ][line width=0.08]  [draw opacity=0] (13.4,-6.43) -- (0,0) -- (13.4,6.44) -- (8.9,0) -- cycle    ;
    \draw [color={rgb, 255:red, 74; green, 144; blue, 226 }  ,draw opacity=1 ][line width=1.5]  [dash pattern={on 1.69pt off 2.76pt}]  (374.24,156.35) -- (440.79,107.27) ;
    \draw [shift={(444,104.89)}, rotate = 143.59] [fill={rgb, 255:red, 74; green, 144; blue, 226 }  ,fill opacity=1 ][line width=0.08]  [draw opacity=0] (13.4,-6.43) -- (0,0) -- (13.4,6.44) -- (8.9,0) -- cycle    ;
    \draw [color={rgb, 255:red, 74; green, 144; blue, 226 }  ,draw opacity=1 ][line width=1.5]  [dash pattern={on 1.69pt off 2.76pt}]  (140.01,132.96) -- (206.56,83.87) ;
    \draw [shift={(209.78,81.5)}, rotate = 143.59] [fill={rgb, 255:red, 74; green, 144; blue, 226 }  ,fill opacity=1 ][line width=0.08]  [draw opacity=0] (13.4,-6.43) -- (0,0) -- (13.4,6.44) -- (8.9,0) -- cycle    ;
    \draw [color={rgb, 255:red, 74; green, 144; blue, 226 }  ,draw opacity=1 ][line width=1.5]  [dash pattern={on 1.69pt off 2.76pt}]  (292.01,118.93) -- (358.56,69.84) ;
    \draw [shift={(361.78,67.46)}, rotate = 143.59] [fill={rgb, 255:red, 74; green, 144; blue, 226 }  ,fill opacity=1 ][line width=0.08]  [draw opacity=0] (13.4,-6.43) -- (0,0) -- (13.4,6.44) -- (8.9,0) -- cycle    ;
    \draw [color={rgb, 255:red, 74; green, 144; blue, 226 }  ,draw opacity=1 ][line width=1.5]  [dash pattern={on 1.69pt off 2.76pt}]  (209.78,81.5) -- (357.79,67.83) ;
    \draw [shift={(361.78,67.46)}, rotate = 174.72] [fill={rgb, 255:red, 74; green, 144; blue, 226 }  ,fill opacity=1 ][line width=0.08]  [draw opacity=0] (13.4,-6.43) -- (0,0) -- (13.4,6.44) -- (8.9,0) -- cycle    ;
    \draw [color={rgb, 255:red, 74; green, 144; blue, 226 }  ,draw opacity=1 ][line width=1.5]  [dash pattern={on 1.69pt off 2.76pt}]  (361.78,67.46) -- (509.79,53.8) ;
    \draw [shift={(513.77,53.43)}, rotate = 174.72] [fill={rgb, 255:red, 74; green, 144; blue, 226 }  ,fill opacity=1 ][line width=0.08]  [draw opacity=0] (13.4,-6.43) -- (0,0) -- (13.4,6.44) -- (8.9,0) -- cycle    ;
    \draw [color={rgb, 255:red, 74; green, 144; blue, 226 }  ,draw opacity=1 ][line width=1.5]  [dash pattern={on 1.69pt off 2.76pt}]  (444,104.89) -- (510.55,55.8) ;
    \draw [shift={(513.77,53.43)}, rotate = 143.59] [fill={rgb, 255:red, 74; green, 144; blue, 226 }  ,fill opacity=1 ][line width=0.08]  [draw opacity=0] (13.4,-6.43) -- (0,0) -- (13.4,6.44) -- (8.9,0) -- cycle    ;
    \draw [color={rgb, 255:red, 74; green, 144; blue, 226 }  ,draw opacity=1 ][line width=1.5]  [dash pattern={on 1.69pt off 2.76pt}]  (526.23,142.32) -- (592.78,93.23) ;
    \draw [shift={(596,90.86)}, rotate = 143.59] [fill={rgb, 255:red, 74; green, 144; blue, 226 }  ,fill opacity=1 ][line width=0.08]  [draw opacity=0] (13.4,-6.43) -- (0,0) -- (13.4,6.44) -- (8.9,0) -- cycle    ;
    \draw [color={rgb, 255:red, 74; green, 144; blue, 226 }  ,draw opacity=1 ][line width=1.5]  [dash pattern={on 1.69pt off 2.76pt}]  (444,104.89) -- (592.02,91.22) ;
    \draw [shift={(596,90.86)}, rotate = 174.72] [fill={rgb, 255:red, 74; green, 144; blue, 226 }  ,fill opacity=1 ][line width=0.08]  [draw opacity=0] (13.4,-6.43) -- (0,0) -- (13.4,6.44) -- (8.9,0) -- cycle    ;
    \draw [color={rgb, 255:red, 74; green, 144; blue, 226 }  ,draw opacity=1 ][line width=1.5]  [dash pattern={on 1.69pt off 2.76pt}]  (209.78,81.5) -- (276.33,32.41) ;
    \draw [shift={(279.55,30.04)}, rotate = 143.59] [fill={rgb, 255:red, 74; green, 144; blue, 226 }  ,fill opacity=1 ][line width=0.08]  [draw opacity=0] (13.4,-6.43) -- (0,0) -- (13.4,6.44) -- (8.9,0) -- cycle    ;
    \draw [color={rgb, 255:red, 74; green, 144; blue, 226 }  ,draw opacity=1 ][line width=1.5]  [dash pattern={on 1.69pt off 2.76pt}]  (361.78,67.46) -- (428.33,18.37) ;
    \draw [shift={(431.55,16)}, rotate = 143.59] [fill={rgb, 255:red, 74; green, 144; blue, 226 }  ,fill opacity=1 ][line width=0.08]  [draw opacity=0] (13.4,-6.43) -- (0,0) -- (13.4,6.44) -- (8.9,0) -- cycle    ;
    \draw [color={rgb, 255:red, 74; green, 144; blue, 226 }  ,draw opacity=1 ][line width=1.5]  [dash pattern={on 1.69pt off 2.76pt}]  (279.55,30.04) -- (427.56,16.37) ;
    \draw [shift={(431.55,16)}, rotate = 174.72] [fill={rgb, 255:red, 74; green, 144; blue, 226 }  ,fill opacity=1 ][line width=0.08]  [draw opacity=0] (13.4,-6.43) -- (0,0) -- (13.4,6.44) -- (8.9,0) -- cycle    ;
    \draw (172.32,139.37) node [anchor=north west][inner sep=0.75pt]  [color={rgb, 255:red, 74; green, 144; blue, 226 }  ,opacity=1 ] [align=left] {$\displaystyle\bf \Pi $};
    \end{tikzpicture}
    \caption{A rank-2 lattice $\Lambda$ and its fundamental parallelepiped $\Pi$.}
    \label{fig:a rank 2 lattice}
\end{figure}
\noindent Dividing the fundamental parallelogram for $\Lambda$ into halves along a diagonal yields two congruent triangles, the interiors of which have empty intersection with $\Lambda$. This warrants the following definition. 
\begin{definition}[Special Triangle, Triangulation]
    A {\it special triangle} in $\Lambda$ is a non-degenerate triangle $T$ whose vertices are integral points and such that $|\overset{\circ}{T}\cap\Lambda|=0$.
    A {\it triangulation} of a polygon $P\subset V$ is a planar graph $\Delta$ embedded in $V$ such that
    \begin{enumerate}[label={\roman*.}\;]
        \item $\Delta\subset P$;
        \item $\partial P\subset\Delta$;
        \item $P\setminus\Delta$ is a union of disjoint interiors of triangles.
    \end{enumerate}
    We call each complimentary component of $P\setminus\Delta$ a {\it face} of $\Delta$. A triangulation is called {\it special} if each face of $\Delta$ is the interior of a special triangle. Note that a special triangulation $\Delta$ of $P$ satisfies $(P\setminus\Delta)\cap\mathbb{Z}^k=\emptyset$.
\end{definition}

Clearly, the special triangles obtained by dividing a fundamental parallelogram along a diagonal each have area $\frac{det(\Lambda)}{2}$. We claim that all special triangles in $\Lambda$ have this property.
\begin{lemma}
    Given a rank-2 integer lattice $\Lambda$ in $\mathbb{R}^3$, let $T$ be a special triangle in $\Lambda$. Then 
    \[
        Area(T)=\frac{det(\Lambda)}{2}.
    \]
\end{lemma}
\begin{proof}
    Assume that one of the vertices of $T$ is the origin. Let $\vec{u}$, $\vec{v}$, and $\vec{u}-\vec{v}$ denote sides of $T$ written as position vectors. We will show that the parallelogram $P$ determined by $\vec{u}-\vec{v}$ and $\vec{v}$ satisfies $|P\cap\mathbb{Z}^3|=4$ and is therefore a fundamental parallelogram for $\Lambda$ by Lemma \ref{lem:simple is fundamental}. Clearly $|\partial P\cap\mathbb{Z}^3|=4$. Thus we only need to show that $\overset{\circ}{P}\cap\mathbb{Z}^3=\emptyset$.
    
    Now, $\vec{u}$ is a diagonal of $P$ and divides $P$ into two congruent triangles, one of which is $T$. Let us call the other triangle $T'$. Suppose $T'$ contains the lattice vector $\vec{\lambda}$ in its interior. The transformation  $\vec{w}\mapsto\vec{u}-\vec{w}$ maps $T'$ bijectively onto $T$. The image of $\vec{\lambda}$, $\vec{u}-\vec{\lambda}$ is again in $\Lambda$. But $\vec{u}-\vec{\lambda}\in\overset{\circ}{T}$ by bijectivity, a contradiction. 
\end{proof}
\noindent{\bf Euler's Formula.} Recall that Euler's formula relates the number of vertices ($v$), edges ($e$), and faces ($f$) of a triangulation $\Delta$ of a polygon living in the plane via the equation $f-e+v=1$.
\medskip

The following result is a version of Pick's Theorem adapted for rank-2 lattices in $\mathbb{R}^3$. It provides a crucial step in Proof 2 of Theorem \ref{thm:A}, and thus provides a means for constructing a special class of Ehrhart polynomials. We include a proof for completeness. Our proof utilizes Euler's formula and can be considered traditional. The technique appears in \cite{coxeter1989introduction}, and a version adapted to $\mathbb{R}^2$ can be found in \cite{AhnDouglas2025}.
\begin{lemma}\label{lem:Picks in 3D}
    Let $\Lambda$ be a rank 2 lattice in $\mathbb{R}^3$. Given a polygon $P$ in $\mathbb{R}^3$ with vertices in $\Lambda$, 
    \[
        |P\cap\mathbb{Z}^3|=\frac{Area(P)}{det(\Lambda)}+\frac{|\partial P\cap\mathbb{Z}^3|}{2}+1.
    \]
\end{lemma}
\begin{proof}
    First, we subdivide $P$ into special triangles, each having area $\frac{1}{2}det(\Lambda)$. Let $f$, $e$, and $v$ denote the numbers of faces, edges, and vertices, respectively, in $\Delta$. Then $f$ equals the number of special triangles in $P$. That is,
    \[
        f=2\frac{Area(P)}{det(\Lambda)}.
    \]
    Next, note that twice the number of special triangles in $P$ equals twice the number of edges in $\Delta$ minus the number of edges in $\partial P\cap\Delta$. In other words,
    \[
        6\frac{Area(P)}{det(\Lambda)}=2e-|\partial P\cap\mathbb{Z}^3|.
    \]
    Thus, the numbers of edges and vertices are
    \[
        e=3\frac{Area(P)}{det(\Lambda)}+\frac{|\partial P\cap\mathbb{Z}^3|}{2}  \hspace{1cm}\&\hspace {1cm}v=|P\cap\mathbb{Z}^3|,
    \]
    respectively. We can now apply Euler's Formula.
    \[
        1=f-e+v=2\frac{Area(P)}{det(\Lambda)}-3\frac{Area(P)}{det(\Lambda)}-\frac{|\partial P\cap\mathbb{Z}^3|}{2}+|P\cap\mathbb{Z}^3|
    \]
    Rearranging gives the desired result.
\end{proof}

As a special case of Lemma \ref{lem:Picks in 3D}, let $\Lambda$ be the standard lattice $\mathbb{Z}^2$ in the $xy$-plane. Then a fundamental parallelogram for $\Lambda$ is given by the unit square. Hence $det(\Lambda)=1$. Additionally, for any polygon $P$ in the $xy$-plane, $|P\cap\mathbb{Z}^3|$ and $|\partial P\cap\mathbb{Z}^3|$ reduce to $|P\cap\mathbb{Z}^2|$ and $|\partial P\cap\mathbb{Z}^2|$, respectively. Making the necessary substitutions and rearranging, we recover Pick's original theorem.
\begin{equation}\label{eqn:Picks}
    Area(P)=\left|\overset{\circ}{P}\cap\mathbb{Z}^2\right|+\frac{|\partial P\cap\mathbb{Z}^2|}{2}-1
\end{equation}
This is equivalent to the original formulation in \cite{pick1899geometrisches}. Pick's original proof however, did not rely directly on Euler's formula. They instead used a similar combinatorial technique by gluing triangles together and applying the inclusion-exclusion principle.
\medskip

\noindent{\bf Computing $det(\Lambda)$.}\;
Let $V\subset\mathbb{R}^3$ be the vector space with planar equation $ax+by+cz=0$, for integers $a$, $b$, and $c$, and let $\Lambda=V\cap\mathbb{Z}^3$. We compute $det(\Lambda)$ by constructing bases vectors $\vec{v}_1$ and $\vec{v}_2$ and computing the area of the fundamental parallelogram with sides $\vec{v}_1$ and $\vec{v}_2$. Constructing $\vec{v}_1$ is easy enough, let
\[
    \vec{v}_1=\left\langle\frac{b}{\gcd(a,b)},\;\frac{-a}{\gcd(a,b)},\;0\right\rangle.
\]
For the second basis vector, we let
\[
    \vec{v}_2=\left\langle\frac{c}{\gcd(a,b,c)}s,\;\frac{c}{\gcd(a,b,c)}t,\;-\gcd\left(\frac{a}{\gcd(a,b,c)},\frac{b}{\gcd(a,b,c)}\right)\right\rangle,
\]
where 
\[
    \frac{a}{\gcd(a,b,c)}s+\frac{b}{\gcd(a,b,c)}t=\gcd\left(\frac{a}{\gcd(a,b,c)},\frac{b}{\gcd(a,b,c)}\right).
\]
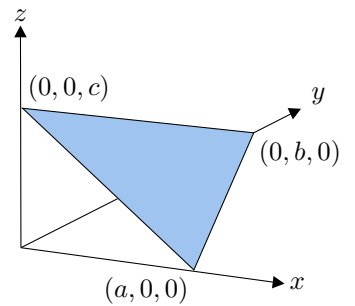
\begin{wrapfigure}{r}{0.33\textwidth}
    \begin{center}
    \begin{tikzpicture}[x=0.5pt,y=0.55pt,yscale=-1,xscale=1]
    \draw    (218.01,185.39) -- (217.66,36.35) ;
    \draw [shift={(217.65,33.35)}, rotate = 89.87] [fill={rgb, 255:red, 0; green, 0; blue, 0 }  ][line width=0.08]  [draw opacity=0] (8.93,-4.29) -- (0,0) -- (8.93,4.29) -- cycle    ;
    \draw    (393.16,106.8) -- (426.01,91.6) ;
    \draw [shift={(428.74,90.34)}, rotate = 155.16] [fill={rgb, 255:red, 0; green, 0; blue, 0 }  ][line width=0.08]  [draw opacity=0] (8.93,-4.29) -- (0,0) -- (8.93,4.29) -- cycle    ;
    \draw    (413.36,209.63) -- (218.01,185.39) ;
    \draw [shift={(416.33,210)}, rotate = 187.07] [fill={rgb, 255:red, 0; green, 0; blue, 0 }  ][line width=0.08]  [draw opacity=0] (8.93,-4.29) -- (0,0) -- (8.93,4.29) -- cycle    ;
    \draw  [fill={rgb, 255:red, 74; green, 144; blue, 226 }  ,fill opacity=0.53 ] (219.21,89.99) -- (393.16,106.8) -- (348.12,200.6) -- cycle ;
    \draw    (218.01,185.39) -- (290.65,151.93) ;
    \draw (217.65,31.73) node [anchor=south] [inner sep=0.75pt]   [align=left] {$\displaystyle z$};
    \draw (434.61,88.72) node [anchor=south west] [inner sep=0.75pt]   [align=left] {$\displaystyle y$};
    \draw (418.33,210) node [anchor=west] [inner sep=0.75pt]   [align=left] {$\displaystyle x$};
    \draw (346.12,203.6) node [anchor=north east] [inner sep=0.75pt]   [align=left] {$\displaystyle ( a,0,0)$};
    \draw (395.16,109.8) node [anchor=north west][inner sep=0.75pt]   [align=left] {$\displaystyle ( 0,b,0)$};
    \draw (221.21,86.99) node [anchor=south west] [inner sep=0.75pt]   [align=left] {$\displaystyle ( 0,0,c)$};
    \end{tikzpicture}
    \end{center}
    \caption{An axial triangle.}
\end{wrapfigure}
We leave it to the reader to verify that $\{\vec{v}_1,\vec{v}_2\}$ is indeed a basis. One strategy is to show that the triangle with sides $\vec{v}_1$, $\vec{v}_2$, and $\vec{v}_1-\vec{v}_2$ is special. Computing $|\vec{v}_1\times\vec{v}_2|$ gives
\begin{equation}\label{eqn:det of Lambda}
    det(\Lambda)=\frac{\sqrt{a^2+b^2+c^2}}{\gcd(a,b,c)}.
\end{equation}

Lastly, we introduce the main object of study in Chapter \ref{sec:main results}.
\begin{definition}
    We call a triangle $T$ in $\mathbb{R}^3$ {\it axial} if it has integral vertices with exactly one vertex contained on each axis.
\end{definition}
Let $T$ be an axial triangle. The vertices of $T$ have the form $(a,0,0)$, $(0,b,0)$, and $(0,0,c)$, for some nonzero $a,b,c\in\mathbb{Z}$. It follows that $T$ is necessarily contained in an octant of $\mathbb{R}^3$. Note that any two octants can be exchanged through a series of rigid transformations of $\mathbb{R}^3$ that preserve the lattice $\mathbb{Z}^3$. Therefore, we may assume that $a$, $b$, and $c$ are all positive, thus placing $T$ in the first octant.
    
\subsection{Restricted Partition Functions}

First, we establish some notation and useful formulas. For a more in-depth analysis, consult \cite{beck2015computing}. Let $x$ be a real number.  The {\it fractional part} of $x$ is denoted by $\{x\}$ and defined by
\[ 
    \{x\}=x-\lfloor x\rfloor,
\]
where the {\it flooring function} of $x$ is defined by $\lfloor x\rfloor=max\{n\in\mathbb{Z}\,|\,n\le x\}$. We will also need the {\it ceiling function}, $\lceil x\rceil=\lfloor x\rfloor+1$, for $x\notin\mathbb{Z}$. If $x\in\mathbb{Z}$, we declare $\lceil x\rceil=x$. We will often employ the following formula for the flooring function. Given a positive integer $b$ and a nonnegative integer $a$, 
\[
    \lfloor a/b\rfloor=\frac{a-a\mod{b}}{b}.
\]

We collect a few more identities and facts about fractional parts and flooring functions that will be of use to us later on.  
\begin{lemma}\label{lem: fractional facts}
    Let $n,a,b\in\mathbb{Z}_{>0}$ where $b\nmid a$. Then
    \[
        \left\{n-\frac{a}{b}\right\}=1-\left\{\frac{a}{b}\right\}.
    \]
    For positive integers $i$ and $j$ with $i\equiv j\mod{\frac{b}{\gcd(a,b)}}$, we have
    \[
        \left\{\frac{ia}{b}\right\}=\left\{\frac{ja}{b}\right\}.
    \]
\end{lemma}

Lastly, a well known formula used to compute compounded flooring:
\begin{equation}\label{eqn:compounded flooring}
    \left\lfloor\frac{\lfloor a/b\rfloor}{c}\right\rfloor=\left\lfloor\frac{a}{bc}\right\rfloor.
\end{equation}

We now describe the main tool used in our first proof of Theorem \ref{thm:A}.
\begin{definition}
    Given a positive integer $n$ and a set of positive integers $A=\{a_1,\,a_2,\,\ldots,\,a_k\}$, we define the {\it restricted partition function} of $n$ with respect to $A$ by
    \[
        P_{A}(n)=\#\left\{(x_1,\,x_2,\,\ldots,\,x_k)\in\mathbb{Z}_{\ge0}^k\;|\;a_1x_1+a_2x_2+\cdots+a_kx_k=n\right\}.
    \]
\end{definition}
In other words, $P_A(n)$ counts the number of solutions in nonnegative integers to the equation
\[
    a_1x_1+a_2x_2+\cdots+a_kx_k=n.
\]
$P_{\{a_1,\ldots,a_k\}}(n)$ is often referred to elsewhere in the literature as the {\sl denumerant of $a_1,\ldots,a_k$}.

The following is known as {\it Popoviciu's Theorem.} The original statement and proof are found in \cite{popoviciu1953asupra}. 
\begin{lemma}\label{lem:P of a,b}
    Let $a,b\in\mathbb{Z}_{>0}$ with $d=\gcd(a,b)$. Then
    \begin{equation*}
        P_{\{a,b\}}(n)=\begin{cases}
        \frac{nd}{ab}-\left\{\frac{\beta n}{a}\right\}-\left\{\frac{\alpha n}{b}\right\}+1 & {\rm if}\;\;d|n\\
        \hspace{.5cm}0 & {\rm otherwise},
        \end{cases}
    \end{equation*}
    where $\beta\frac{b}{d}\equiv1\mod{\frac{a}{d}}$ and $\alpha\frac{a}{d}\equiv1\mod{\frac{b}{d}}$.
\end{lemma}
More generally, we have the following recursive formula (\cite[Exercise 1.27]{beck2015computing}).
\begin{equation}\label{eqn:recursive}
    P_{\{a_1,\ldots,a_k\}}(n)=\sum_{i\ge0}P_{\{a_1,\ldots,a_{k-1}\}}(n-ia_k),
\end{equation}
where, by convention, $P_{A}(n)=0$ if $n<0$.
\medskip

Finally, consider an axial triangle $T$ with vertices $(a,0,0)$, $(0,b,0)$, and $(0,0,c)$, where $a,b,c\in\mathbb{Z}_{>0}$. An equation for the plane containing $T$ is then
\begin{equation}\label{eqn:T}
        bcx+acy+abz=abc,
\end{equation}
and a connection between restricted partition functions and integral points on polyhedra is immediately seen by
\[
    P_{\{bc,ac,ab\}}(abc)=|T\cap\mathbb{Z}^3|.
\]
Furthermore, since a point $(x,y,z)$ is contained in $\partial T$ if and only if at least one of its coordinates is 0, the number of integral points in $\overset{\circ}{T}$ equals the number of solutions in positive integers to Equation \ref{eqn:T}. 

\subsection{Quasi-Polynomials and Ehrhart Theory}

We call a function $q(n)$ {\it periodic with period $p$} if $q(n+p)=q(n)$ for all $n$, and $p$ is the minimum positive integer with this property. Note that if $q$ has period $p$, then $q(n+kp)=q(n)$ for any integers $k$ and $n$. Similarly, if $q(n+r)=q(r)$ for all $n$, then $p$ divides $r$. For example, let $q(n)=\gcd(a,bn)$ for relatively prime positive integers $a$ and $b$. We claim the period of $q$ is $a$. Indeed, 
\[
    q(n+a)=\gcd(a,bn+ab)=\gcd(a,bn)=q(n).
\]
Thus, the period of $q$ divides $a$. Let $p$ be the period of $q$. Then $\gcd(a,bn+bp)=\gcd(a,bn)$ for all $n\in\mathbb{Z}$. Setting $n=0$ yields $\gcd(a,bp)=\gcd(a,0)=a$. But $\gcd(a,b)=1$ and thus $\gcd(a,bp)=\gcd(a,p)=p$ implying $p=a$.

A {\it quasi-polynomial} is a function of the form
\[
    q(t)=a_{n}(t)t^n+a_{n-1}(t)t^{n-1}+\cdots+a_1(t)t+a_0(t),
\]
where the coefficients $a_i(t)$ are themselves periodic functions of $t$. If $a_i(t)$ has period $p_i$, then $q(t)$ has period ${\rm lcm}(p_0,\ldots,p_n)$. More generally, if $Q_1,\ldots,Q_k$ are quasi-polynomials with periods $P_1,\ldots,P_k$, respectively, then the period of $Q_1(t)+\cdots+Q_k(t)$ is ${\rm lcm}(P_1,\ldots,P_k)$.

We now provide a brief summary of Ehrhart's celebrated theorem for rational polyhedra and give simple examples. We begin with a polyhedron $P\subset\mathbb{R}^k$ whose vertices are the terminal points of vectors $\vec{v}_1,\ldots,\vec{v}_n\in\mathbb{R}$. The standard way of scaling $P$ amounts to multiplying the vectors $\vec{v}_1,\ldots,\vec{v}_n$ by a scalar matrix whose diagonal entries all equal $t$, for some $t\in\mathbb{Z}_{>0}$. The result is a polyhedron, $tP$, that has been {\it scaled by $t$}. In \cite{ehrhart1962polyedres}, it is shown that, for polyhedra whose vertices contain rational coordinates, the function $\mathcal{L}(t,P)=|tP\cap\mathbb{Z}^k|$ is evaluated by a quasi-polynomial. However, if the terminal points of $\vec{v}_1,\ldots,\vec{v}_n$ are all integral points, then $\mathcal{L}(t,P)$ is evaluated by a polynomial.
\begin{example}\label{example:quasi polynomial}
First, consider the unit square $S$ with vertices $(0,0)$, $(1,0)$, $(0,1)$, and $(1,1)$. Scaling $S$ by $t$ yields a square with vertices $(0,0)$, $(t,0)$, $(0,t)$, and $(t,t)$. We can apply Pick's Theorem (Equation \ref{eqn:Picks}) to enumerate the number of points on $tS$:
\[
    \mathcal{L}(t,S)=|tS\cap\mathbb{Z}^2|=t^2+\frac{4t}{2}+1.
\]
If we instead scale the rectangle $R$ with vertices $(0,0)$, $(1,0)$, $(0,\frac{1}{2})$, and $(1,\frac{1}{2})$ by $t$, we find
\[
    \mathcal{L}(t,R)=\frac{1}{2}t^2+\left(\frac{3-t\;\text{mod}\,2}{2}\right)t+1-\frac{t\;\text{mod}\,2}{2},
\]
a quasi-polynomial with period 2.
\end{example}

More generally, Ehrhart in \cite{ehrhart1962polyedres} analyzed the generating function $E(z)=1+\sum_{t=1}^\infty\mathcal{L}(t,P)z^t$, for a polyhedron $P$. The {\it Ehrhart quasi-polynomial} is then obtained by analyzing the roots of $E(z)$.
\section{Main Results}\label{sec:main results}

\subsection{Ehrhart Polynomials of Triangles in \texorpdfstring{$\mathbb{R}^3$}{}}\label{subsec:axial triangles}

We begin this section by providing two proofs of Theorem \ref{thm:A}. The first proof is a direct application of Popoviciu's Theorem (Lemma \ref{lem:P of a,b}) that avoids the heavy machinery commonly used to derive the same result. Our second proof uses Pick's Theorem adapted to rank-2 integer lattices in $\mathbb{R}^3$ (Lemma \ref{lem:Picks in 3D}). Both proofs apply foundational results found in standard texts. See \cite{barvinok2008integer} and \cite{beck2015computing}.
\begin{theorem}\label{thm:Clara1}
    Given $a,b,c\in\mathbb{Z}_{>0}$, 
    \[
        P_{\{bc,ac,ab\}}(abc)=\frac{\gcd(bc,ac,ab)}{2}+\frac{{\gcd(a,b)+\gcd(a,c)+\gcd(b,c)}}{2}+1.
    \]
\end{theorem}
\begin{proof}[Proof 1.]
    Using Equation \ref{eqn:recursive},
    \begin{equation*}
        P_{\{bc,ac,ab\}}(abc) = \overset{c}{\underset{i=0}{\sum}} P_{\{bc,ac\}}(abc - iab).
    \end{equation*}   
    \noindent Now, for the allotted values of $i$, $P_{\{bc,ac\}}(abc - iab) = 0$ if and only if $\gcd(bc,ac) \nmid abc- iab$.
    \noindent Assume that $\gcd(ac,bc) \nmid abc -iab$. Clearly, $\gcd(ac,bc)$ divides $abc$. So, it must be the case that $c\gcd(a,b) \nmid iab$. Dividing both sides by $\gcd(a,b)$ gives us $c \nmid i\,\textrm{lcm}(a,b)$.
    Let $\delta = \gcd(c, \textrm{lcm}(a,b))$. Then, $\frac{c}{\delta} \nmid i\frac{\textrm{lcm}(a,b)}{\delta}$ can be simplified to $\frac{c}{\delta} \nmid i$.
    To summarize, if $\frac{c}{\delta} \nmid i$, then $P_{\{bc,ac,ab\}}(abc)=0$. So, by assessing what values for $i$ allow $\frac{c}{\delta} | i$, we can find all values where $P_{\{bc,ac,ab\}}(abc)\neq0$. Clearly, $\frac{c}{\delta} | i$ for $i = 0, \frac{c}{\delta}, 2\frac{c}{\delta}, 3\frac{c}{\delta},...,\delta\frac{c}{\delta}$. Thus we have 
    $P_{\{bc,ac,ab\}}(abc)= \overset{\delta}{\underset{i=0}{\sum}} P_{\{bc,ac\}}(abc-i\frac{abc}{\delta})$.
    
    Let $d=\gcd(b,a)$. Then $\gcd(ac,bc)=cd$. Next, we apply Lemma \ref{lem:P of a,b}, which allows us to write\\ $\overset{\delta}{\underset{i=0}{\sum}} P_{\{bc,ac\}}\left(abc-i\frac{abc}{\delta}\right)$ as

    \begin{equation}\label{expr:pop} 
        \overset{\delta}{\underset{i=0}{\sum}} \left( \frac{\left(abc-i\frac{abc}{\delta}\right)cd}{bc(ac)} - \left\{ \frac{\beta\left(abc - i\frac{abc}{\delta}\right)}{bc}\right\} - \left\{ \frac{\alpha\left(abc - i\frac{abc}{\delta}\right)}{ac}\right\} +1 \right).
    \end{equation}
    where $\beta\frac{ac}{cd}\equiv1\,\text{mod}\,{\frac{bc}{cd}}$ and $\alpha\frac{bc}{cd}\equiv1\,\text{mod}\,\frac{ac}{cd}.$

    First, let us focus on the first term of Expression \ref{expr:pop}. 
    \begin{align*}
        \overset{\delta}{\underset{i=0}{\sum}} \frac{(abc-i\frac{abc}{\delta})cd}{bc(ac)} &= \overset{\delta}{\underset{i=0}{\sum}}\left(1-\frac{i}{\delta}\right) \gcd(b,a) \\
        &= \gcd(a,b)\left(\overset{\delta}{\underset{i=0}{\sum}}1-\overset{\delta}{\underset{i=0}{\sum}}\frac{i}{\delta}\right)\\
        &= \gcd(a,b)(\delta+1)-\frac{\gcd(a,b)(\delta+1)}{2}\\
        &= \frac{\gcd(a,b)(\delta+1)}{2}.
    \end{align*}
    Recall that $\delta = \gcd(c,\text{lcm}(a,b))$. Simplifying further we have,
    $$\frac{\gcd(a,b)(\delta +1)}{2} = \frac{\gcd(ac, bc, ab)}{2} + \frac{\gcd(a,b)}{2}.$$
    Thus, we have handled the first term of Expression \ref{expr:pop}.

    Now, we will consider the second term: $\overset{\delta}{\underset{i=0}{\sum}}\left\{ \frac{\beta\left(abc - i\frac{abc}{\delta}\right)}{bc}\right\}$. Note the definition of the fractional part of $\frac{\beta n}{a}$, 
    \begin{align*}
        \left \{\frac{\beta n}{a}\right\}&=\frac{\beta n}{a}- \left\lfloor\frac{\beta n}{a} \right\rfloor \\
        &= \frac{\beta n}{a} - \left( \frac{\beta n - \beta n\,\text{mod}\,a}{a}\right).
    \end{align*}
    Using this definition we can rewrite the second term of Expression \ref{expr:pop}.

    \begin{align*}
        \overset{\delta}{\underset{i=0}{\sum}}\left\{ \frac{\beta\left(abc - i\frac{abc}{\delta}\right)}{bc}\right\} &= \overset{\delta}{\underset{i=0}{\sum}}\left\{ \frac{\beta abc}{bc}- \frac{\beta \frac{abc}{\delta}i}{bc} \right\} \\
        &= \overset{\delta}{\underset{i=0}{\sum}} \left\{ \beta a - \frac{\beta a i}{\delta}\right\} .
    \end{align*}
     Note that whenever $\beta a - \frac{\beta ai}{\delta}$ is an integer, $\left\{ \beta a - \frac{\beta a i}{\delta}\right\} = 0$. So, we first break the sum into $\gcd(\beta a , \delta)$ sums, then we omit all terms that are equal to zero and apply the second portion of Lemma 2.8. This gives us,
    {\footnotesize\begin{align}\label{EQN:beakingsum}
        \overset{\delta}{\underset{i=0}{\sum}} \left\{ \beta a - \frac{\beta a i}{\delta}\right\}  &= \overset{\frac{\delta}{\gcd(\beta a, \delta)}}{\underset{i=0}{\sum}} \left\{ \beta a - \frac{\beta a i}{\delta}\right\} +\overset{2\frac{\delta}{\gcd(\beta a, \delta)}}{\underset{i=\frac{\delta}{\gcd(\beta a, \delta)}}{\sum}} \left\{ \beta a - \frac{\beta a i}{\delta}\right\}+\cdots+\overset{\delta}{\underset{i=\delta - \frac{\delta}{\gcd(\beta a, \delta)}}{\sum}} \left\{ \beta a - \frac{\beta a i}{\delta}\right\}\nonumber\\
        = &\overset{\frac{\delta}{\gcd(\beta a, \delta)}-1}{\underset{i=1}{\sum}} \left(1 - \left\{ \frac{\beta ai}{\delta}\right\}\right) +\overset{2\frac{\delta}{\gcd(\beta a, \delta)}-1}{\underset{i=\frac{\delta}{\gcd(\beta a, \delta)}+1}{\sum}} \left(1 - \left\{ \frac{\beta ai}{\delta}\right\}\right) +\cdots+\overset{\delta-1}{\underset{i=\delta - \frac{\delta}{\gcd(\beta a, \delta)}+1}{\sum}} \left(1 - \left\{ \frac{\beta ai}{\delta}\right\}\right).
    \end{align}}
    Since $i\equiv i+k\frac{\delta}{\gcd(\beta a,\delta)}\mod\frac{\delta}{\gcd(\beta a,\delta)}$ for any integer $k$, we can apply Lemma \ref{lem: fractional facts} and obtain
    \begin{equation*}
         \overset{\delta}{\underset{i=0}{\sum}} \left\{ \beta a - \frac{\beta a i}{\delta}\right\}=\gcd(\delta,\beta a)\sum_{i=1}^{\frac{\delta}{\gcd(\beta a, \delta)}-1}\left(1 - \left \{ \frac{\beta ai}{\delta}\right \}\right).
    \end{equation*}
    Isolating the sum, we have
    \begin{equation*}
        \overset{\frac{\delta}{\gcd(\beta a, \delta)}-1}{\underset{i=1}{\sum}} 1 - \left \{ \frac{\beta ai}{\delta}\right \}
        = \overset{\frac{\delta}{\gcd(\beta a, \delta)}-1}{\underset{i=1}{\sum}} 1 - \frac{\beta ai\, \text{mod}\,\delta}{\delta}.
    \end{equation*}
    Now we factor out $\frac{-1}{\delta}$ from the second term so that we can work to further reduce the expression.
    \begin{align*}
        \sum_{i=1}^{\frac{\delta}{\gcd(\beta a, \delta)}-1} 1 - \frac{\beta ai\, \text{mod}\,\delta}{\delta}
        &=\sum_{i=1}^{\frac{\delta}{\gcd(\beta a, \delta)}-1} 1-\frac{1}{\delta}\sum_{i=1}^{\frac{\delta}{\gcd(\beta a, \delta)}-1} \beta ai \, \text{mod}\,\delta \\
        &= \frac{\delta}{\gcd(\beta a, \delta)}-1 - \left( \frac{\gcd(\beta a , \delta)}{\gcd(\beta a , \delta)}\right)\frac{1}{\delta}\sum_{i=1}^{\frac{\delta}{\gcd(\beta a, \delta)}-1} \beta ai\,\text{mod}\,\delta \\
        &= \frac{\delta}{\gcd(\beta a, \delta)}-1 - \gcd(\beta a,\delta)\frac{1}{\delta}\sum_{i=1}^{\frac{\delta}{\gcd(\beta a, \delta)}-1} \frac{\beta ai}{\gcd(\beta a ,\delta)}\, \text{mod}\,\frac{\delta}{\gcd(\beta a , \delta)}.
    \end{align*}
    Since $\frac{\beta ai}{\gcd(\beta a ,\delta)}$ and $\text{mod}\frac{\delta}{\gcd(\beta a , \delta)}$ are relatively prime, this allows us to write
    \[
        \sum_{i=1}^{\frac{\delta}{\gcd(\beta a, \delta)}-1}\frac{\beta ai}{\gcd(\beta a ,\delta)}\, \text{mod}\,\frac{\delta}{\gcd(\beta a , \delta)}=\sum_{i=1}^{\frac{\delta}{\gcd(\beta a, \delta)}-1}i=\frac{\left(\frac{\delta}{\gcd(\beta a ,\delta)}-1\right)\frac{\delta}{\gcd(\beta a ,\delta)}}{2}.
    \]
    Bringing it all together, 
    \begin{align*}
        \sum_{i=0}^\delta\left\{ \frac{\beta\left(abc - i\frac{abc}{\delta}\right)}{bc}\right\} &= \gcd(\beta a,\delta)\left( \frac{\delta}{\gcd(\beta a, \delta)}-1-\frac{\gcd(\beta a,\delta)}{\delta}\frac{\left(\frac{\delta}{\gcd(\beta a ,\delta)}-1\right)\frac{\delta}{\gcd(\beta a ,\delta)}}{2}\right)\\
        &=\frac{\delta}{2} - \frac{\gcd(\beta a,\delta)}{2}.
    \end{align*}
    
    For the third term of \ref{expr:pop}, $\overset{\delta}{\underset{i=0}{\sum}}\left\{ \frac{\alpha\left(abc - i\frac{abc}{\delta}\right)}{ac}\right\}$, the process will mirror the above and become $\frac{\delta}{2} - \frac{\gcd(\alpha b,\delta)}{2}$.

    Next, we must remove $\beta$ and $\alpha$ from the expressions $\frac{\delta}{2} - \frac{\gcd(\beta a,\delta)}{2}$ and $\frac{\delta}{2} - \frac{\gcd(\alpha b,\delta)}{2}$, respectively. We will do so by showing that $\gcd(\delta, \beta a) =\gcd(c,a)$ and $\gcd(\delta, \alpha b) =\gcd(c,b)$. First,
    recall that $\delta = \gcd(\text{lcm}(a,b),c)$ and the definitions of $\beta$, $\alpha$, and $d$ in our case are:
    \begin{align*}
        \beta \frac{ac}{cd} &\equiv   1\,\text{mod}\,\frac{bc}{cd} \hspace{2cm}  \alpha \frac{bc}{cd} \equiv   1\,\text{mod}\,\frac{ac}{cd} \hspace{2cm} 
        d = \gcd(a,b).
    \end{align*}
    Also, we know that $\beta \frac{ac}{cd} \equiv   1\text{mod}\frac{bc}{cd}$ implies that $\beta \frac{ac}{cd} = 1+ k\frac{bc}{cd}$ for some $k \in \mathbb{Z}$.
    \\
    
    Now, note that $\gcd(\delta, \beta a)=\gcd(\gcd(c,\text{lcm}(a,b)), \beta a)=\gcd(\gcd(c,(\text{lcm}(a,b),\beta a) $. Since $\beta a \equiv d \, \text{mod}\, b$, we have $\beta a = d + kb$ for some $k \in \mathbb{Z}$. Therefore,
    \begin{align*}
        \gcd(\delta, \beta a) &= \gcd\left(c,\gcd\left(\frac{ab}{d}, d+kb\right)\right) \\ 
        &= \gcd\left(c,\gcd\,d\left(\frac{a}{d}\frac{b}{d}, 1+k\frac{b}{d}\right)\right).
    \end{align*}
    We can see that $1+k\frac{b}{d}$ and $\frac{b}{d}$ are relatively prime. So,
    \begin{align*}
        \gcd\left(c,\gcd\,d\left(\frac{a}{d}\frac{b}{d}, 1+k\frac{b}{d}\right)\right) &= \gcd\left(c,\gcd\,d\left(\frac{a}{d}, 1+k\frac{b}{d}\right)\right) \\
        &= \gcd\left(c,\gcd\,d\left(\frac{a}{d}, \frac{\beta a }{d}\right)\right) \\
        &= \gcd\left(c,\gcd\,\left(a, \beta a\right)\right) \\
        &= \gcd(c,a).
    \end{align*}
    \noindent
    Thus, we have shown that $\gcd(\delta, \beta a) = \gcd(c,a)$, and by the same methods, $\gcd(\delta,\alpha b) = \gcd(c,b)$.
    \\
    Next, we can substitute these values into $\frac{\delta}{2}-\frac{\gcd(\beta a, \delta)}{2}$ and $\frac{\delta}{2}-\frac{\gcd(\alpha b, \delta)}{2}$, giving us $\frac{\delta}{2}-\frac{\gcd(c,a)}{2}$ and $\frac{\delta}{2}-\frac{\gcd(c,b)}{2}$, respectively. 

    Clearly, the final term of Expression \ref{expr:pop}, $\overset{\delta}{\underset{i=0}{\sum}} 1 = \delta +1$.

    Now that we have addressed all of the terms in \ref{expr:pop}, we can see the following:
    \begin{align*}
         P_{\{bc,ac,ab\}}(abc) 
         &= \sum_{i=0}^\delta\left(\frac{(abc-i\frac{abc}{\delta}) \gcd(bc,ac)}{bc(ac)} - \left\{ \frac{\beta(abc - i\frac{abc}{\delta})}{bc}\right\} - \left\{ \frac{\alpha(abc - i\frac{abc}{\delta})}{ac}\right\} +1 \right)\\
         &= \frac{\gcd(ac, bc, ab)}{2} + \frac{\gcd(a,b)}{2} - \left(\frac{\delta}{2}-\frac{\gcd(c,a)}{2}\right)-\left(\frac{\delta}{2}-\frac{\gcd(c,b)}{2}\right)+\delta +1 \\ 
         &= \frac{\gcd(ac, bc, ab)}{2} + \frac{\gcd(a,b)}{2}+\frac{ \gcd(c,a)}{2}+ \frac{\gcd(c,b)}{2} +1.
    \end{align*}
\end{proof}
\begin{proof}[Proof 2.]
    We perform the following transformations on $\mathbb{R}^3$. Shift by $-a$ along the $x$-axis, then reflect in the $yz$-plane. Given positive integers $a,b,c$, this has the effect of taking the plane $bcx+acy+abz=abc$ to the vector space $V$ with planar equation $bcx-acy-abz=0$. Let $\Lambda=V\cap\mathbb{Z}^3$. Since the transformations just described preserve the lattice $\mathbb{Z}^3\subset\mathbb{R}^3$, it follows that a special triangle on the plane $bcx+acy+abz=abc$ has area $det(\Lambda)/2$. 

    Now, $P_{\{bc,ac,ab\}}(abc)$ equals $|T\cap\mathbb{Z}^3|$, where $T$ is the triangle with vertices $(a,0,0)$, $(0,b,0)$, and $(0,0,c)$. If $T'$ denotes the image of $T$ under the transformations above then $|T\cap\mathbb{Z}^3|=|T'\cap\mathbb{Z}^3|$, $|\partial T\cap\mathbb{Z}^3|=|\partial T'\cap\mathbb{Z}^3|$, and $Area(T)=Area(T')$. Then, by Lemma \ref{lem:Picks in 3D},
    \begin{equation}\label{eqn:proof 2}
        |T\cap\mathbb{Z}^3|=\frac{Area(T)}{det(\Lambda)}+\frac{|\partial T\cap\mathbb{Z}^3|}{2}+1.
    \end{equation}
    By Equation \ref{eqn:det of Lambda}, 
    \[
        det(\Lambda)=\frac{\sqrt{(bc)^2+(ac)^2+(ab)^2}}{\gcd(bc,ac,ab)}.
    \]
    To get the area of $T$, we compute half the magnitude of
    \[
        \langle-a, b, 0\rangle\times\langle-a, 0, c\rangle=\langle bc,\; ac,\; ab\rangle.
    \]
    This gives
    \[
        Area(T)=\frac{\sqrt{(bc)^2+(ac)^2+(ab)^2}}{2}.
    \]
    Finally, note that the boundary of $T$ consists of three line segments, one in each coordinate plane. It follows that 
    \[
        |\partial T\cap\mathbb{Z}^3|=\gcd(a,b)+\gcd(a,c)+\gcd(b,c).
    \]
    The desired formula now follows after performing the necessary substitutions in Equation \ref{eqn:proof 2}.
\end{proof}
As a special case, let $a$, $b$, and $c$ be positive integers. The axial triangle with vertices $(bc,0,0)$, $(0,ac,0)$, and $(0,0,ab)$ lies on the plane with equation $ax+by+cz=abc$, and Theorem \ref{thm:Clara1} allows us to compute $P_{\{a,b,c\}}(abc)$. More generally, we have the following, which appears as Corollary \ref{cor:B} in the Introduction.
\begin{corollary}\label{cor:MainResult}
    Given $a,b,c\in\mathbb{Z}_{>0}$, let $l={\rm lcm}(a,b,c)$. For any nonnegative integer $k$,
    \[
        P_{\{a,b,c\}}(kl)=k^2\frac{\gcd(l^2/bc,l^2/ac,l^2/ab)}{2}+k\left(\frac{\gcd(l/a,l/b)+\gcd(l/a,l/c)+\gcd(l/b,l/c)}{2}\right)+1.
    \]
\end{corollary}
\begin{proof}
    Let $A=k\frac{l}{a}$, $B=k\frac{l}{b}$, and $C=k\frac{l}{c}$. First note that $(x,y,z)$ is a solution to
    \[
        BCx+ACy+ABz=ABC
    \]
    if and only if it is a solution to $ax+by+cz=kl$. The desired expression then follows from computing
    \begin{align*}
        \frac{\gcd(BC,AC,AB)}{2}+\frac{{\gcd(A,B)+\gcd(A,C)+\gcd(B,C)}}{2}+1.
    \end{align*}
\end{proof}
\begin{remark}\label{rmk:Ehrhart Polynomials}
    We now describe the Ehrhart Polynomial corresponding to an axial triangle $T$ with vertices $(a,0,0)$, $(0,b,0)$, and $(0,0,c)$ by applying Theorem \ref{thm:Clara1} to the axial triangle with vertices $(ta,0,0)$, $(0,tb,0)$, and $(0,0,tc)$:
    \[
        \mathcal{L}(t,T)=P_{\{t^2bc,t^2ac,t^2ab\}}(t^3abc)= t^2\frac{\gcd(bc,ac,ab)}{2}+t\left(\frac{\gcd(a,b)+\gcd(a,c)+\gcd(b,c)}{2}\right)+1.
    \]
    We note however that Ehrhart's method for deriving such expressions was quite different. In a nutshell, Ehrhart equated the generating function $E(z)$ with expanding about the poles of certain rational expressions. Whereas our Proof 1 of Theorem \ref{thm:A} used mostly modular arithmetic.
\end{remark}

We end this section with a brief investigation of a useful property of Ehrhart polynomials often referred to as {\it Ehrhart-Macdonald reciprocity}:
\[
    \mathcal{L}(t,\overset{\circ}{P})=(-1)^d\mathcal{L}(-t,P),
\]
where $d$ is the dimension of $P$. When $d$ is even, this simplifies to $\mathcal{L}(t,\overset{\circ}{P})=\mathcal{L}(-t,P)$.

For the axial triangle $T$, the vertices of $tT$ are $(ta,0,0)$, $(0,tb,0)$, and $(0,0,tc)$, and $|\partial tT\cap\mathbb{Z}^3|=t(\gcd(a,b)+\gcd(a,c)+\gcd(b,c))$. It follows that
\begin{align*}
    \mathcal{L}(t,\overset{\circ}{T})=|\overset{\circ}{tT}\cap\mathbb{Z}^3|&=|tT\cap\mathbb{Z}^3|-|\partial tT\cap\mathbb{Z}^3| \\
    &=t^2\frac{\gcd(bc,ac,ab)}{2}-t\frac{\gcd(a,b)+\gcd(a,c)+\gcd(b,c)}{2}+1\\
    &=\mathcal{L}(-t,T),
\end{align*}
verifying the reciprocity formula for axial triangles. Recall that $|\overset{\circ}{tT}\cap\mathbb{Z}^3|$ counts the number of solutions in positive integers to $bcx+acy+abz=tabc$.
\begin{example}\label{example:2}
    Let $t\in\mathbb{Z}_{>0}$ and $S$ denote the  axial triangle with vertices $(1,0,0)$, $(0,1,0)$, and $(0,0,1)$. As in Example \ref{example:1},
    \[
        \mathcal{L}(t,S)=\frac{t^2}{2}+\frac{3t}{2}+1.
    \]
    This equals $\binom{t+2}{2}$, the number of ways of choosing 2 objects from $t+2$ objects. By reciprocity,
    \[
        \mathcal{L}(t,\overset{\circ}{S})=\frac{t^2}{2}-\frac{3t}{2}+1=\binom{t-1}{2}.
    \]
\end{example}
Examples \ref{example:1} and \ref{example:2} illustrate a special case of the classic {\sl stars and bars} problem, \cite{feller1950introduction}. In a nutshell, Feller was interested in determining how many ways $n$ stars could be placed into $r$ cells. For a schematic, they used $r-1$ vertical bars to separate the cells. The formula
\[
   \binom{n+r-1}{r-1} 
\]
determines the desired number of ways. If each cell is required to contain at least one star then the number of ways of distributing the $n$ stars is given by
\[
    \binom{n-1}{r-1}.
\]
Example \ref{example:2} resolves the case when $r=3$.

\subsection{Bounds on \texorpdfstring{$P_{\{a,b,c\}}(n)$}{}}\label{subsec:bounds}

In this section, we attempt to apply the same method of proof as Proof 1 of Theorem \ref{thm:A} in the more general setting. While we are unable to arrive at a closed formula that works in general, we discover relatively simple upper and lower bounds on $P_{\{a,b,c\}}(n)$, Theorem \ref{thm:C}. We will first describe an easy-to-obtain upper bound.
\medskip

If we make no assumptions between $n$ and $a$, $b$, or $c$, it is possible for the plane with equation $ax+by+cz=n$ to intersect the axes at points with rational coordinates. However, the triangle $T$ with coordinates $(n/a,0,0)$, $(0,n/b,0)$, and $(0,0,n/c)$ still satisfies $P_{\{a,b,c\}}(n)=|T\cap\mathbb{Z}^3|$. 

Even if $T$ has vertices which are not integral points, the projection of $T$ to any of the three coordinate planes is injective. It follows that any integral point on $T$ corresponds to a unique integral point in the projection to any of the three coordinate planes. We can use this fact to find an upper bound on $P_{\{a,b,c\}}(n)$.
Specifically, suppose $a\le b\le c$. Then projecting the triangle $T$ with vertices $(n/a,n/b,n/c)$ onto the $yz$-plane yields the least upper bound. If $P$ denotes this projection, then $|P\cap\mathbb{Z}^2|$ gives an upper bound for $|T\cap\mathbb{Z}^3|$. For a concrete formula, we apply Pick's Theorem to the triangle with vertices $(0,0)$, $(\lceil n/b\rceil,0)$, and $(0,\lceil n/c\rceil)$. This yields
\[
    |T\cap\mathbb{Z}^3|\le \frac{\lceil n/b\rceil\lceil n/c\rceil}{2}+\frac{\lceil n/b\rceil+\lceil n/c\rceil+\gcd\left(\lceil n/b\rceil,\lceil n/c\rceil\right)}{2}+1.
\]
Equality holds when $a=0$, $b|n$, and $c|n$.
\medskip

Before proving Theorem \ref{thm:C}, first note that if $\gcd(a,b,c)=d>1$ then $ax+by+cz=n$ has solutions if and only if $d$ divides $n$. We can thus assume $\gcd(a,b,c)=1$. We can further reduce to the case that $a$, $b$, and $c$ are pairwise relatively prime. Let $g_1=\gcd(b,c)$, $g_2=\gcd(a,c)$, and $g_3=\gcd(a,b)$. We denote by $a'$, $b'$, and $c'$ the inverses of $a$ modulo $g_1$, $b$ modulo $g_2$, and $c$ modulo $g_3$, respectively. With $n_1=na'\mod g_1$, $n_2=nb'\mod g_2$, and $n_3=nc'\mod g_3$, we define the substitutions
\[
    A = \frac{a}{g_2 g_3}, \hspace{1cm} B = \frac{b}{g_3 g_1}, \hspace{1cm} C = \frac{c}{g_1 g_2}.
\]
Then $(x_0,y_0,z_0)$ is a solution to $ax+by+cz=n$ if and only if $\left(\frac{x_0-n_1}{g_1},\frac{y_0-n_2}{g_2},\frac{z_0-n_3}{g_3}\right)$ is a solution to $Ax+By+Cz=N$, where
$$N = \frac{n - a n_1 - b n_2 - c n_3}{g_1 g_2 g_3}.$$
The precise derivation can be found in ~\cite[Section 2.1]{Binner2021Combinatorial}. Thus, there is no loss in generality in assuming $a$, $b$, and $c$ are pairwise relatively prime.
\begin{proof}[Proof of Theorem \ref{thm:C}]
    Let $a,b,c\in\mathbb{Z}_{>0}$ be pairwise relatively prime. Then
    \[
        P_{\{a,b,c\}}(n)=\sum_{i\ge0}P_{\{a,b\}}(n-ic).
    \]
    Since $\gcd(a,b)=1$, $P_{\{a,b\}}(n-ic)$ is nonzero for each $i=0,\ldots,\lfloor n/c\rfloor$. Thus,
    \[
        P_{\{a,b,c\}}(n)=\sum_{i=0}^{\lfloor n/c\rfloor}P_{\{a,b\}}(n-ic).
    \]
    Applying Lemma \ref{lem:P of a,b} yields
    \[
        P_{\{a,b,c\}}(n)=\sum_{i=0}^{\lfloor n/c\rfloor}\left(\frac{n-ic}{ab}-\left\{\frac{\beta(n-ic)}{a}\right\}-\left\{\frac{\alpha(n-ic)}{b}\right\}+1\right),
    \]
    where $\beta b\equiv1\mod a$ and $\alpha a\equiv1\mod b$. Many of the terms are easy to combine.
    \begin{equation*}
        \sum_{i=0}^{\lfloor n/c\rfloor}\left(\frac{n-ic}{ab}+1\right)=\left(\left\lfloor\frac{n}{c}\right\rfloor+1\right)\left(\frac{n+n\mod c}{2ab}+1\right).
    \end{equation*}
    Consider the sums
    \[
        \sum_{i=0}^{\lfloor n/c\rfloor}\left\{\frac{\beta(n-ic)}{a}\right\}\hspace{.5cm}\&\hspace{.5cm}\sum_{i=0}^{\lfloor n/c\rfloor}\left\{\frac{\alpha(n-ic)}{b}\right\}.
    \]
    We have
    \begin{equation*}
        \sum_{i=0}^{\lfloor n/c\rfloor}\left\{\frac{\beta(n-ic)}{a}\right\}=\sum_{i=0}^{\lfloor n/c\rfloor}\left(\frac{\beta(n-ic)}{a}-\left\lfloor\frac{\beta(n-ic)}{a}\right\rfloor\right)
        =\sum_{i=0}^{\lfloor n/c\rfloor}\left(\frac{\beta(n-ic)\mod a}{a}\right).
    \end{equation*}
    Let $m=\lfloor n/c\rfloor$, and write $m=qa+r_a$, for $0\le r_a<a$ and $q>0$. Then
    \[
        \sum_{i=0}^m\beta(n-ic)\mod a=\sum_{i=0}^{aq-1}\beta(n-ic)\mod a+\sum_{i=aq}^{aq+r_a}\beta(n-ic)\mod a.
    \]
    Now,
    \[
        \sum_{i=0}^{aq-1}\beta(n-ic)\mod a=q\sum_{i=1}^{a-1}i=q\frac{a(a-1)}{2}=(\lfloor n/c\rfloor-\lfloor n/c\rfloor\mod a)\frac{a-1}{2}.
    \]
    Thus,
    \[
        \sum_{i=0}^{\lfloor n/c\rfloor}\left\{\frac{\beta(n-ic)}{a}\right\}=\frac{a-1}{2}\left\lfloor\frac{\lfloor n/c\rfloor}{a}\right\rfloor+\frac{1}{a}\sum_{i=0}^{r_a}\beta(n-ic)\mod a.
    \]
    Similarly, we can write $m=q'b+r_b$, for $0\le r_b<b$ and $q'>0$. Then
    \[
        \sum_{i=0}^{\lfloor n/c\rfloor}\left\{\frac{\alpha(n-ic)}{b}\right\}=\frac{b-1}{2}\left\lfloor\frac{\lfloor n/c\rfloor}{b}\right\rfloor+\frac{1}{b}\sum_{i=0}^{r_b}\alpha(n-ic)\mod b.
    \]
    Now, the sum
    \[
        \sum_{i=0}^{r}\beta(n-ic)\mod a
    \]
    is always greater than or equal to 
    \[
        \sum_{i=0}^{r_a}i=\frac{r_a(r_a+1)}{2}
    \]
    and less than or equal to
    \[
        \sum_{i=0}^{r_a}a-1-i=(r_a+1)(a-1)-\frac{r_a(r_a+1)}{2}=(r_a+1)\frac{2a-r_a-2}{2}.
    \]
    Similarly,
    \[
        \frac{r_b(r_b+1)}{2}\le\sum_{i=0}^{r_b}\alpha(n-ic)\mod b\le(r_b+1)\frac{2b-r_b-2}{2}.
    \]
    For convenience in notation, given $a,b,c,n\in\mathbb{Z}_{>0}$, we define
    \begin{equation}\label{eqn:upper bound 1}
        N_{\{a,b,c\}}(n)=\left(\left\lfloor\frac{n}{c}\right\rfloor+1\right)\left(\frac{n+n\mod c}{2ab}+1\right).
    \end{equation}
    An upper bound for $P_{\{a,b,c\}}(n)$ is then
    \[
        U_{\{a,b,c\}}(n)=N_{\{a,b,c\}}(n)-\left(\frac{a-1}{2}\left\lfloor\frac{n}{ac}\right\rfloor+\frac{r_a(r_a+1)}{2a}\right)-\left(\frac{b-1}{2}\left\lfloor\frac{n}{bc}\right\rfloor+\frac{r_b(r_b+1)}{2b}\right).
    \]
    And, a lower bound for $P_{\{a,b,c\}}(n)$ is given by
    \[
        L_{\{a,b,c\}}(n)=N_{\{a,b,c\}}(n)-\left(\frac{a-1}{2}\left\lfloor\frac{n}{ac}\right\rfloor+(r_a+1)\frac{2a-r_a-2}{2a}\right)-\left(\frac{b-1}{2}\left\lfloor\frac{n}{bc}\right\rfloor+(r_b+1)\frac{2b-r_b-2}{2b}\right).
    \]

    Now, let $Gap(n)=U_{\{a,b,c\}}(n)-L_{\{a,b,c\}}(n)$. Then
    \[
        Gap(n)=(r_a+1)\frac{a-r_a-1}{a}+(r_b+1)\frac{b-r_b-1}{b}.
    \]
    Letting $x=r_a+1$, the curve of $y=\frac{1}{a}x(a-x)$ is a parabola with a maximum value achieved when $x=a/2$. Similarly, the maximum value of $y=\frac{1}{b}x(b-x)$ is achieved when $x=b/2$. Thus,
    \begin{equation}\label{eqn:gap bound}
        Gap(n)\;\le\;(a/2)\frac{a-a/2}{a}+(b/2)\frac{b-b/2}{b}\;=\;\frac{a+b}{4}.
    \end{equation}

    Lastly, to show that $U_{\{a,b,c\}}(n)$ and $L_{\{a,b,c\}}(n)$ both have period $abc$, observe that $N_{\{a,b,c\}}(n)$ has period $c$, $\frac{a-1}{2}\left\lfloor\frac{n}{ac}\right\rfloor$ has period $ac$, $\frac{b-1}{2}\left\lfloor\frac{n}{bc}\right\rfloor$ has period $bc$, while $\frac{r_a(r_a+1)}{2a}$ and $(r_a+1)\frac{2a-r_a-2}{2a}$ both have period $a$, and $(r_b+1)\frac{2b-r_b-2}{2b}$ and $\frac{r_b(r_b+1)}{2b}$ both have period $b$. It follows that the periods of $U_{\{a,b,c\}}(n)$ and $L_{\{a,b,c\}}(n)$ are both ${\rm lcm}(a,b,c,ac,bc)={\rm lcm}(ac,bc)$. And, since $a$, $b$, and $c$ are pairwise relatively prime, lcm$(ac,bc)=c\,{\rm lcm}(a,b)=abc$.
\end{proof}

Observe that when $a$, $b$, and $c$ are relatively prime and $n$ is a multiple of $abc$, $P_{\{a,b,c\}}(n)=U_{\{a,b,c\}}(n)$. Indeed, in this case, $r_a=r_b=0$, $n=kabc$ for some $k\in\mathbb{Z}_{>0}$, and the expression for $U_{\{a,b,c\}}(n)$ reduces to
\begin{align*}
    U_{\{a,b,c\}}(kabc)&=\left(kab+1\right)\left(\frac{kc}{2}+1\right)-\left(\frac{kab-kb+kab-ka}{2}\right)\\
    &=k^2\frac{abc}{2}+k\frac{a+b+c}{2}+1,
\end{align*}
which agrees with Corollary \ref{cor:MainResult}.
\begin{example}
    Given $N\in\mathbb{Z}_{>0}$, the equation $5x+10y+25z=N$ has solutions if and only if $5|N$. Ignoring the cases with no solutions, we are able to simplify by dividing by $5$ yielding $x+2y+5z=n$, where $n=\frac{N}{5}$. Performing the necessary substitutions and simplifying, 
    $$U_{\{1,2,5\}}(n) = \left( \left\lfloor \frac{n}{5} \right\rfloor + 1 \right) \left( \frac{n + (n \bmod 5)}{4} + 1 \right) - \frac{1}{2}\left( \left\lfloor\frac{n}{10}\right\rfloor + \left\lfloor\frac{n}{5}\right\rfloor \bmod 2 \right)$$
    and
    $$L_{\{1,2,5\}}(n) = \left( \left\lfloor \frac{n}{5} \right\rfloor + 1 \right) \left( \frac{n + (n \bmod 5)}{4} + 1 \right) - \frac{1}{2}\left( \left\lfloor\frac{n}{10}\right\rfloor + 1 \right).$$
    In this case,
    $$\text{Gap}(n) = \frac{1 - \left(\left\lfloor \frac{n}{5}\right\rfloor\bmod 2\right)}{2}.$$
    Thus, whenever $\left\lfloor \frac{n}{5}\right\rfloor$ is odd, $U_{\{1,2,5\}}(n)$ and $L_{\{1,2,5\}}(n)$ both equal $P_{\{1,2,5\}}(n)$. When$\left\lfloor \frac{n}{5}\right\rfloor$ is even, $P_{\{1,2,5\}}(n)$ is given by whichever of $U_{\{1,2,5\}}(n)$ or $L_{\{1,2,5\}}(n)$ is an integer.

    More generally, when $a=1$ and $b=2$, we have $\text{Gap}(n)\le\frac{3}{4}$. In this case, $P_{\{1,2,c\}}(n)$ is easily obtained by calculating either $U_{\{1,2,c\}}(n)$ or $L_{\{1,2,c\}}(n)$ and then rounding appropriately. 
\end{example}
\begin{remark}
    Equation \ref{eqn:upper bound 1} yields another upper bound on $P_{\{a,b,c\}}(n)$. The bound is sharp when both $a$ and $b$ divide both $n$ and $c$.
\end{remark}
\begin{remark}\label{rmk:Frob number}
    For pairwise relatively prime positive integers $a$, $b$, and $c$, let $N=N_{\{a,b,c\}}(n)$. The {\it $k$th Frobenius number}, a concept introduced in \cite{BeRo2004}, for $a$, $b$, and $c$ is the largest integer $n$ such that $P_{\{a,b,c\}}(n)=k$ ($k=0$ is the classical case). Then, for $n$ to be the $k$th Frobenius number of $a$, $b$, and $c$, by applying Equation \ref{eqn:gap bound} we must have
    \begin{equation}\label{eqn:Frob bounds}
        k+R_a+R_b+\frac{a+b}{4}-N\ge0,
    \end{equation}
    where $R_a=\frac{a-1}{2}\left\lfloor\frac{n}{ac}\right\rfloor+\frac{r_a(r_a+1)}{2a}$ and $R_b=\frac{b-1}{2}\left\lfloor\frac{n}{bc}\right\rfloor+\frac{r_b(r_b+1)}{2b}$. Now, $N$ is quadratic as a function of $n$, whereas all remaining terms are either linear functions of $n$ or bounded. It follows that the $k$th Frobenius number exists for any triple $a,b,c$ of pairwise relatively prime positive integers.
    
    This is by no means a new result. In fact, a number of bounds have been developed and exist in the literature. Recently, algorithms for computing upper and lower bounds can be found in \cite{TRIPATHI2017368}. Improvements on the algorithms contained therein are provided by \cite{suhajda2026frobeniusnumbervariables}. The upper bound obtained via Equation \ref{eqn:Frob bounds} is looser than those found in the aforementioned references. 
\end{remark}

\section{Future Work}

\begin{enumerate}[label={\roman*.}\;]
    \item The upper and lower bounds from Theorem \ref{thm:C} apply when we have three terms $a,b,c$. We expect that the bounds can be recursively extended to give bounds in the case of arbitrarily many terms $a_1,\ldots,a_k$. We ask whether this can be accomplished in such a way that the difference between the upper and lower bounds remains uniformly bounded? Once such bounds are achieved, an analysis of the algorithmic complexity required to compute them is warranted, as well as a comparison with algorithms currently in use to compute $P_A(n)$.
    \item The closed formulas computed in Section \ref{subsec:axial triangles} apply to very specific triangles in $\Lambda$. Can they be generalized to give closed formulas for $|P\cap\mathbb{Z}^3|$ for a more general class of polygons in $\Lambda$? 
\end{enumerate}

\bibliographystyle{amsalpha}
\bibliography{name}
\end{document}